\documentclass[preprint,11pt]{article}%
\usepackage{graphicx}
\usepackage{amsmath,amsxtra,amssymb,latexsym, amscd,amsthm}
\usepackage[mathscr]{eucal}
\usepackage{floatflt}
\usepackage{bbm}
\newtheorem{prop}{Proposition}[section]
\theoremstyle{definition}

\newtheorem{assum}{Assumption}[section]
\theoremstyle{remark}
\newtheorem{rem}{Remark}[section]
\numberwithin{equation}{section}

\theoremstyle{plain}
{\theoremstyle{plain}\newtheorem{theorem}{Theorem}[section]}
{\theoremstyle{plain}\newtheorem{proposition}{Proposition}[section]}
{\theoremstyle{plain}\newtheorem{lemma}{Lemma}[section]}
{\theoremstyle{plain}}
\numberwithin{equation}{section}

\begin{document}
	
\title
{Caputo fractional stochastic differential equations: Lipschitz continuity in the fractional order}
\author{Ta Cong Son\thanks{Department of Mathematics, VNU University of Science, Vietnam National University, Hanoi, 334 Nguyen
Trai, Thanh Xuan, Hanoi, 084 Vietnam.}\and Nguyen Tien Dung$^\ast$\thanks{Corresponding author. Email: dung@hus.edu.vn}\and Pham Thi Phuong Thuy\thanks{The faculty of Basic Sciences, Vietnam Air Defence and Air Force Academy, Son Tay, Ha Noi, 084 Vietnam.}\and Tran Manh Cuong$^\ast$ \and Hoang Thi Phuong Thao$^\ast$\and Pham Dinh Tung$^\ast$}
	
\date{\today}          
	
	\maketitle
	\begin{abstract}
In this paper, we consider a  class of the Caputo fractional stochastic differential equations of fractional order $\alpha \in (\frac{1}{2},1]$. Our aim is to analyze of the continuous dependence of solutions on the fractional order $\alpha.$ We first provide explicit estimates for the rate of weak convergence the solutions. We then describe the exact asymptotic behavior of this convergence to show that the rate is optimal.
	\end{abstract}	

	\section{Introduction}
In recent decades, the fractional calculus and fractional stochastic differential equations have attracted great attention of researchers due to their coverage for a great variety of applications in the real world such as dynamics of complex systems in engineering, fluid mechanics as well as financial and biological models, and so on. By using different definitions of fractional operators (such as Caputo
derivative, Caputo-Fabrizio derivative, Riemann-Liouville derivative and integral, Atangana-Baleanu derivative, etc), many types of fractional stochastic differential equations have been introduced. We refer the reader to monographs \cite{BK,Georgiev2018,KS} for more details. In this paper, we consider Caputo fractional stochastic differential equations of the form
\begin{align}\label{ca2}
D_{0^{+}}^\alpha X_{\alpha,t} = b(t,X_{\alpha,t}) + \sigma(t,X_{\alpha,t})\frac{dB_t}{dt},\,\,0\leq t\leq T,
\end{align}
where the fractional order $\alpha \in (\frac{1}{2},1],$ $B= (B_t)_{t\in [0,T]}$ is a standard Brownian motion and $b,\sigma:[0,T]\times \mathbb{R}\to \mathbb{R}$ are measurable functions. Here we write $X_{\alpha,t}$ to stress the dependence of solutions on $\alpha.$ Note that Caputo fractional operator $D_{0^{+}}^\alpha$ is non-local, which strongly makes the equation (\ref{ca2}) suitable and
efficient to describe the long memory or non-local effects characterizing most physical phenomena, see \cite{KD,Oldham2016,Pod} and the references therein.

By its definition, a stochastic process $X_\alpha=(X_{\alpha,t})_{t\in [0,T]}$ is called a solution of the equation ($\ref{ca2}$) with the initial condition $X_{\alpha,0} = x_0\in \mathbb{R}$ if it satisfies the following equation
	\begin{equation}\label{eq1}
	X_{\alpha,t} = x_{0} + \frac{1}{\Gamma(\alpha)} \left(\int_0^t (t-s)^{\alpha -1}b(s,X_{\alpha,s})ds +  \int_0^t (t-s)^{\alpha -1} \sigma(s,X_{\alpha,s})dB_s\right),\,0\leq t\leq T,
	\end{equation}
	where the second integral is interpreted as an It\^o stochastic integral and $\Gamma(\alpha) = \int_0^\infty x^{\alpha -1}e^{-x}dx$ is the Gamma function.

In the last years, several fundamental properties of the solutions to the equation (\ref{eq1}) have been studied by various authors, see e.g. \cite{Guo2021,SonDT,SonDT2020,Xu2019,Wang1,Wang2022}. We note that the equation ($\ref{ca2}$) belongs to the class of stochastic Volterra equations with singular kernels. For this class, the convergence of solutions with respect to the kernels plays an important role in many applications. We refer the reader to \cite{Abi2019a,Abi2019b,Abi2021,Alfonsi2024} and references therein for several fruitful results. In this research theme, the results obtained for the equation ($\ref{ca2}$) can be summarized as follows: Given $\beta\in (\frac{1}{2},1],$ let $(X_{\beta,t})_{t\in [0,T]}$ solve the following equation
\begin{equation}\label{eq1qc}
	X_{\beta,t} = x_{0} + \frac{1}{\Gamma(\beta)} \left(\int_0^t (t-s)^{\beta -1}b(s,X_{\beta,s})ds +  \int_0^t (t-s)^{\beta -1} \sigma(s,X_{\beta,s})dB_s\right),\,0\leq t\leq T.
	\end{equation}
In \cite{Huong2023,Wang}, authors proved that
\begin{equation}\label{yfl}
E|X_{\alpha,t}-X_{\beta,t}|^p\to 0  \ \ \mbox{ as}  \ \alpha\to \beta
\end{equation}
for every $0\leq t\leq T$ and $p\geq 2.$ Furthermore, by using Theorem 3.1 in recent paper \cite{Alfonsi2024}, one can get the following estimate for the rate of convergence
\begin{equation}\label{yfla1}
E|	X_{\alpha,t} -X_{\beta,t}|^2 \le Ct^{2(\alpha\wedge\beta)-1}(|\ln t|^2+1)|\alpha-\beta|^2\,\,\,\forall\,\alpha,\beta\in (\frac{1}{2},1],
\end{equation}
where $C$ is a positive constant.

 Let us recall that, in the literature, $L^p$-results of the forms (\ref{yfl}) and (\ref{yfla1}) are the so-called the strong convergence of $X_{\alpha,t}$ to $X_{\beta,t}$ as $\alpha\to \beta.$ On the other hand, from practical point of view, weak convergence results are very useful and have been widely investigated, see e.g. \cite{Kloeden1992}. The main task is to obtain an explicit estimate for the quantity $|Eg(X_{\alpha,t}) -Eg(X_{\beta,t})|,$ where $g$ belongs to a suitable class of test functions. When $g$ is a Lipschitz continuous function, the problem is trivial. Indeed, it follows from  (\ref{yfla1}) that
$$|Eg(X_{\alpha,t}) -Eg(X_{\beta,t})|\leq CE|X_{\alpha,t} -X_{\beta,t}|\le C|\alpha-\beta|\,\,\,\forall\,\alpha,\beta\in (\frac{1}{2},1].$$
The non-trivial case is when the test function $g$ is only measurable and bounded. In the present paper, our aim is to handle this non-trivial case. Our tactics is to use a general result established in our recent paper \cite{Dung2022} by means of Malliavin calculus, see Lemma \ref{dltq} below. We obtain the following estimate
\begin{equation}\label{yflb}
|Eg(X_{\alpha,t}) -Eg(X_{\beta,t})| \le C\|g\|_{\infty}t^{\alpha\wedge\beta-\beta}(|\ln t|+1)|\alpha -\beta|\,\,\,\forall\,\alpha\in \big(\frac{1}{2},1],\beta\in [\frac{7}{8},1],
\end{equation}
where $\|.\|_\infty$ denotes the supremum norm. We also prove the optimality of the estimate (\ref{yflb}) as $\alpha\to \beta$ by showing that the limit $\lim\limits_{\alpha\to \beta}\frac{Eg(X_{\alpha,t}) -Eg(X_{\beta,t})}{\alpha-\beta}$  exists and can be computed explicitly.

The rest of this paper is organized as follows. In Section \ref{sec1}, we provide some useful estimates and  recall some concepts of Malliavin calculus. In Section \ref{kja}, we revisit the strong convergence results. We point out, in Theorem \ref{dl32m} below, that the limit $Y_{\beta,t}:=\lim\limits_{\alpha\to \beta}\frac{X_{\alpha,t} -X_{\beta,t}}{\alpha-\beta}$ exists and can be computed explicitly. Our main results are then stated and proved in Section \ref{kjb}. We provide optimal estimates for the rate of weak convergence in Theorem \ref{dlc}. 
	
	\section{Preliminaries}\label{sec1}

	\subsection{Some useful estimates}	
In this subsection, we collect some fundamental integral estimates and provide a Gronwall-type lemma for singular kernels.
	\begin{lemma}\label{l1} Let  $\Gamma(\alpha) = \int_0^\infty x^{\alpha -1}e^{-x}dx,\alpha>0$ be the Gamma function. Then

\noindent $(i)$ there exits a positive constant $C$ such that $|\Gamma'(\alpha)|+|\Gamma''(\alpha)|\leq C$ for all $\frac{1}{2}<\alpha\leq 1,$

\noindent $(ii)$  there exists $\alpha_\ast\in(1,2)$ such that $\Gamma(\alpha)$ is decreasing in $(0,\alpha_\ast)$ and  increasing in $(\alpha_\ast, \infty).$

	\end{lemma}
	\begin{proof}
$(i)$ It is easy to get
	$$\Gamma'(\alpha) = \int_0^\infty x^{\alpha -1}e^{-x} \ln{x}dx \ \ \mbox{and } \ \ \Gamma''(\alpha) = \int_0^\infty x^{\alpha -1}e^{-x} \ln^2{x}.$$
	Therefore,
		\begin{align*}
		|\Gamma'(\alpha)| &\leq \int_0^\infty x^{\alpha -1}e^{-x} |\ln{x}|dx  =  \int_0^1  x^{\alpha -1}e^{-x}|\ln{x}|dx +  \int_1^\infty  x^{\alpha -1}e^{-x}|\ln{x}|dx \\
		& \leq -\int_0^1 x^{-1/2}\ln{x}dx + \int_1^\infty e^{-x}\ln{x}dx
		=4+\int_1^\infty e^{-x}\ln{x}dx.
		\end{align*}
		and
		\begin{align*}
		|\Gamma''(\alpha)| &=  \int_0^1  x^{\alpha -1}e^{-x}\ln^2{x}dx +  \int_1^\infty  x^{\alpha -1}e^{-x}\ln^2{x}dx \\
		& \leq \int_0^1 x^{-1/2}\ln^2{x}dx + \int_1^\infty e^{-x}\ln^2{x}dx\\&
		=16+\int_1^\infty e^{-x}\ln^2{x}dx.
		\end{align*}
So it holds that
$$|\Gamma'(\alpha)|+|\Gamma''(\alpha)|\leq 20+\int_1^\infty e^{-x}\ln{x}dx+\int_1^\infty e^{-x}\ln^2{x}dx=:C<\infty.$$		
$(ii)$ We observe that $\Gamma(1)=1=\Gamma(2).$ Then there exists $\alpha_\ast\in (1,2)$ such that $\Gamma'(\alpha_\ast)=0.$ On the other hand, $\Gamma''(\alpha)>0 \ \ \mbox{for all} \ \alpha\in (0, \infty).$
So $\Gamma(\alpha)$ is decreasing in $(0,\alpha_\ast)$ and  increasing in $(\alpha_\ast, \infty).$
	\end{proof}
	\begin{lemma}\label{l2}
		For all $\alpha,\beta\in \left(\frac{1}{2},1\right]$. There exists a constant $C>0$ such that
\begin{equation}\label{tfg}
\left|\frac{1}{\Gamma(\alpha)}-\frac{1}{\Gamma(\beta)}\right|\le C|\alpha-\beta| .
\end{equation}
	\end{lemma}
	\begin{proof}
Since $\Gamma(\beta),\Gamma(\alpha)\geq \Gamma(1)=1$ for $\alpha,\beta\in \left(\frac{1}{2},1\right],$ we have
		$$ \bigg|\frac{1}{\Gamma(\alpha)}-\frac{1}{\Gamma(\beta)}\bigg|\leq |\Gamma(\beta)-\Gamma(\alpha)|. $$
So, using the Lagrange theorem, there exists $\xi$ lying between $\alpha$ and $\beta$ such that
$$
\bigg|\frac{1}{\Gamma(\alpha)}- \frac{1}{\Gamma(\beta)}\bigg| \leq |\Gamma'(\xi)(\alpha -\beta)|,
$$
and hence, by the part $(i)$ of Lemma \ref{l1}, we get
$$
\bigg|\frac{1}{\Gamma(\alpha)}- \frac{1}{\Gamma(\beta)}\bigg| \leq C|\alpha -\beta|.
$$
The proof of the lemma is complete.
	\end{proof}
\begin{lemma}\label{lm3}
		Let $\frac{1}{2}<\alpha_0<1.$ For all $\alpha,\beta\in [\alpha_0,1]$, we have
		\begin{align}\label{hh2}
	 \int_0^t\left((t-s)^{\alpha-1}-(t-s)^{\beta-1}\right)^2ds\le C(t^{2\alpha-1}+t^{2\beta-1})(\ln^2t+1)(\alpha-\beta)^2
		\end{align}
		and
			\begin{align}\label{hh2.1}
		\int_0^t\left((t-s)^{\alpha-1}-(t-s)^{\beta-1}\right)^2\ln^2(t-s)ds\le C(t^{2\alpha-1}+t^{2\beta-1})(\ln^4t+1)(\alpha-\beta)^2,
		\end{align}
		where $C$ is a positive constant depending only on $\alpha_0.$
	\end{lemma}
	\begin{proof}
		Put $H_{z}(t,s):=(t-s)^{z-1}$. We have
		$$ \frac{\partial H_{z}(t,s)}{\partial z}=(t-s)^{z-1}\ln(t-s).$$
		By using the Lagrange theorem, we obtain
		\begin{align*}
		H_{\beta}(t,s) -H_{\alpha}(t,s)&=(\beta-\alpha)(t-s)^{\gamma-1}\ln(t-s)
		\end{align*}
		where $\gamma$ is  between $\alpha$ and $\beta$. Hence,
		\begin{align*}
		\int_0^t\big(&(t-s)^{\alpha-1}-(t-s)^{\beta-1}\big)^2ds\\
&= \int_0^t\left(	H_{\beta}(t,s) -H_{\alpha}(t,s)\right)^2ds=(\alpha-\beta)^2\int_0^t(t-s)^{2\gamma-2}\ln^2(t-s)ds\\
&\leq (\alpha-\beta)^2\int_0^t\left((t-s)^{2\alpha-2}+(t-s)^{2\beta-2}\right)\ln^2(t-s)ds\\
		&=(\alpha-\beta)^2\left(t^{2\alpha-1}\int_0^1v^{2\alpha-2}(\ln t+\ln v)^2dv+t^{2\beta-1}\int_0^1v^{2\beta-2}(\ln t+\ln v)^2dv\right)\\
&\leq (\alpha-\beta)^2\left(t^{2\alpha-1}+t^{2\beta-1}\right)\int_0^1v^{2\alpha_0-2}(\ln t+\ln v)^2dv,
		\end{align*}
 which give us the required estimate (\ref{hh2}). The proof of \eqref{hh2.1} can be done similarly. So we omit it.
	\end{proof}
\begin{lemma}\label{ger}Let $v:[0,T]\rightarrow [0,\infty)$ be a real function and $\omega(.)$ is a non-negative continuous function on $[0,T]$ and there are constants $a>0$ and $0<\eta <1$ such that
		$$v(t) \le \omega(t)+ a\int_0^t (t-s)^{\eta-1}v(s)ds,\,\,0\leq t\leq T.$$
		Then, we have
\begin{equation}\label{dhk3}
v(t)\leq 2\omega(t)
+\frac{8at^{\frac{\eta}{2}}}{\eta}\exp\left(\frac{4^{\frac{2}{\eta}}a^{\frac{2}{\eta}}t^2}{\eta^{\frac{2}{\eta}}}\right)
\left(\int_0^t \omega^{\frac{2}{\eta}}(s)ds\right)^{\frac{\eta}{2}},\,\,0\leq t\leq T.
\end{equation}
If, in addition, the function $w$ is non-decreasing, then
\begin{equation}\label{dhk}
v(t)\leq 2\omega(t)\exp\left(\frac{4^{\frac{2}{\eta}}a^{\frac{2}{\eta}}t^2}{\eta^{\frac{2}{\eta}}}\right),\,\,0\leq t\leq T.
\end{equation}
\end{lemma}
\begin{proof} For any $q>\frac{1}{\eta},$ by using H\"older's inequality, we get
\begin{align*}
v^q(t) & \le \left(\omega(t)+ a\int_0^t (t-s)^{\eta-1}v(s)ds\right)^q\\
&\le 2^{q-1}\omega^q(t)+ 2^{p-1}a^q\left(\int_0^t (t-s)^{\eta-1}v(s)ds\right)^q\\
&\leq 2^{q-1}\omega^q(t)+ 2^{p-1}a^q\left(\int_0^t (t-s)^{\frac{q(\eta-1)}{q-1}}ds\right)^{q-1}\int_0^tv^q(s)ds\\
&= 2^{q-1}\omega^q(t)+ 2^{q-1}a^q\left(\frac{q-1}{q\eta-1}\right)^{q-1}t^{q\eta-1}\int_0^tv^q(s)ds,\,\,0\leq t\leq T.
\end{align*}
Fixed $t\in [0,T],$ we have
$$v^q(x) \leq 2^{q-1}\omega^q(x)+ 2^{q-1}a^q\left(\frac{q-1}{q\eta-1}\right)^{q-1}t^{q\eta-1}\int_0^xv^q(s)ds,\,\,0\leq x\leq t.$$
So, by using Gronwall's lemma, we obtain
\begin{multline}\label{bao}
v^q(x)\leq 2^{q-1}\omega^q(x)
+4^{q-1}a^q\left(\frac{q-1}{q\eta-1}\right)^{q-1}t^{q\eta-1}\\
\times\int_0^x \omega^q(s)\exp\left(2^{q-1}a^q\left(\frac{q-1}{q\eta-1}\right)^{q-1}t^{q\eta-1}(x-s)\right)ds,\,\,0\leq x\leq t.
\end{multline}
This, together with the fundamental inequality $(a+b)^{\frac{1}{q}}\leq a^{\frac{1}{q}}+b^{\frac{1}{q}}\,\forall\,a,b\geq 0,$ gives us
\begin{multline*}
v(x)\leq 2\omega(x)\\
+4a\left(\frac{q-1}{q\eta-1}\right)^{\frac{q-1}{q}}t^{\eta-\frac{1}{q}}\exp\left(\frac{2^{q-1}a^q}{q}\left(\frac{q-1}{q\eta-1}\right)^{q-1}t^{q\eta}\right)
\left(\int_0^x \omega^q(s)ds\right)^{\frac{1}{q}},\,\,0\leq x\leq t.
\end{multline*}
Choosing $x=t$ and $q=\frac{2}{\eta}$ yields
$$
v(t)\leq 2\omega(t)
+\frac{8at^{\frac{\eta}{2}}}{\eta}\exp\left(\frac{4^{\frac{2}{\eta}}a^{\frac{2}{\eta}}t^2}{\eta^{\frac{2}{\eta}}}\right)
\left(\int_0^t \omega^{\frac{2}{\eta}}(s)ds\right)^{\frac{\eta}{2}},\,\,0\leq t\leq T.
$$
So the inequality (\ref{dhk3}) is verified. When the function $w$ is non-decreasing, we have $\omega^q(s)\leq \omega^q(x)$ for $s\leq x,$ and hence, (\ref{bao}) reduces to the following
$$v^q(x)\leq 2^{q-1}\omega^q(t)\exp\left(2^{q-1}a^q\left(\frac{q-1}{q\eta-1}\right)^{q-1}t^{q\eta-1}x\right),\,\,0\leq x\leq t.$$
Choosing $x=t$ and $q=\frac{2}{\eta}$ gives us (\ref{dhk}). The proof of the lemma is complete.
\end{proof}
\begin{rem}The Gronwall-type lemma for singular kernels plays an important role in the theory of fractional differential equations. Under the assumption of Lemma \ref{ger}, it was proved in \cite{henry} that
$$v(t) \le \omega(t) + \int_0^t \Big[\sum\limits_{n=1}^\infty \frac{(a\Gamma(\beta))^n}{\Gamma(n\beta)}(t-s)^{n\beta-1}\omega(s)\Big]ds,\,\,0\leq t\leq T.$$
The above estimate has been widely used in the literature. Here, we need to use estimates (\ref{dhk3}) and (\ref{dhk}) for establishing the estimates bounded uniformly in $\alpha.$
\end{rem}
	\subsection{Malliavin calculus}
	Let us recall some elements of Malliavin calculus. We suppose that$(B_t)_{t\in [0, T]}$ is defined on a complete probability space $(\Omega ,\mathcal{F},\mathbb{F},P)$, where $\mathbb{F}= (\mathcal{F}_t)_{t \in [0,T]}$ is a natural filtration generated by the Brownian motion $B.$ For $h \in L^2[0, T]$, we denote by $B(h)$ the Wiener integral $$B(h)= \int_0^Th(t)dB_t.$$
	Let $\mathcal{S}$ denote a dense subset of $L^2(\Omega ,\mathcal{F}, P)$ that consists of smooth random variables of the form
	\begin{align}\label{e4q1}
	F = f(B(h_1),B(h_2),..., B(h_n) ),
	\end{align}
	where $n \in \mathbb{N}, f \in C_0^\infty(\mathbb{R}^n), h_1, h_2, ..., h_n \in L^2[0,T]$. If $F$ has the form (\ref{e4q1}), we define its Malliavin derivative as the process $DF:= {D_tF, t\in [0,T]}$ given by
	$$D_tF =\sum_{k=1}^{n}\frac{\partial f}{\partial x_k}(B(h_1),B(h_2),..., B(h_n))h_k(t).$$
	More generally, for each $k\ge 1,$ we can define the iterated derivative operator on a cylindrical random variable by setting
	$$ D_{t_1,...,t_k}^{k}F=D_{t_1}...D_{t_k}F. $$
	For any $1 \le p,k< \infty$, we denote by $\mathbb{D}^{k,p}$ the closure of $\mathcal{S}$ with respect to the norm
	$$||F||_{k,p}^p:= E|F|^p + E\left[\bigg(\int_0^T|D_uF|^2du\bigg)^{\frac{p}{2}}\right]+...+E\left[\bigg(\int_0^T...\int_0^T|D^{k}_{t_1,...,t_k}F|^2dt_1...dt_k\bigg)^{\frac{p}{2}}\right].$$
A random variable $F$ is said to be Malliavin differentiable if it belongs to $\mathbb{D}^{1,2}$.
	For the convenience of the reader, we recall that a derivative operator $D$ satisfies the chain rule, i.e, $D\phi(F) = \phi'(F)DF$. Furthermore, we have the following relations between Malliavin derivative and the integrals
	$$D_r\bigg( \int_{0}^Tu_sds\bigg)= \int_{r}^TD_ru_sds$$
	and $$D_r\bigg( \int_{0}^Tu_sdB_s\bigg)= u_r+\int_{r}^TD_ru_sdB_s,0 \leq r\leq T, $$
	where $(u_t)_{t\in [0, T]}$ is an $\mathbb{F}$-adapted and Malliavin differentiable stochastic process.
	
	An important operator in the Malliavin's calculus theory is the divergence operator $\delta$, it is the adjoint of the derivative operator $D$. The domain of $\delta$ is the set of all functions $u\in L^2(\Omega, L^2[0,T])$ such that
	$$E|\langle DF, u \rangle_{L^2[0,T]}|\leq C(u)\|F\|_{L^2(\Omega)},$$
	where $C(u)$ is some positive constant depending on $u$. Let $F\in \mathbb{D}^{1,2}$ and $u\in Dom \delta$ such that $Fu\in L^2(\Omega \times[0,T])$ Then $\delta(u)$ is characterized by the following   relationship
	\begin{equation}\label{0jkd3}
	\delta(Fu)=F\delta(u)-\langle DF, u \rangle_{L^2[0,T]}.
	\end{equation}
We have the following general result.
\begin{lemma} \label{dltq}Let $F_1 \in \mathbb{D}^{2,4}$ be such that  $\|DF_1\|_{L^2[0,T]}>0\,\,a.s.$ Then, for any random variable $F_2\in \mathbb{D}^{1,2}$ and any measurable function $g$ with $\|g\|_\infty:=\sup\limits_{x\in \mathbb{R}}|g(x)|\leq 1,$  we have
\begin{align}
&|Eg(F_1)-Eg(F_2)|\notag\\
&\leq C \left(E\|DF_1\|^{-8}_{L^2[0,T]}E\left(\int_0^T\int_0^T|D_\theta D_rF_1|^2d\theta dr\right)^2+(E\|DF_1\|^{-2}_{L^2[0,T]})^2\right)^{\frac{1}{4}}\|F_1-F_2\|_{1,2},\label{uu4}
\end{align}
provided that the expectations exist, where $C$ is an absolute constant.
\end{lemma}
\begin{proof}
This lemma is Theorem 3.1 in our recent paper \cite{Dung2022}. Here we note that the inequality (\ref{uu4}) follows from the relation
\begin{multline}
Eg(F_1)-Eg(F_2)\\=E\left[\int_{F_2}^{F_1} g(z)dz\delta\left(\frac{DF_1}{\|DF_1\|^2_{L^2[0,T]}}\right)\right]-E\left[\frac{g(F_2) \langle DF_1-DF_2 , DF_1 \rangle_{L^2[0,T]}}{\|DF_1\|^2_{L^2[0,T]}}\right].\label{oldl1}
\end{multline}
We also have $E\left[\left(\delta\bigg(\frac{DF_1}{\|DF_1\|^2_{L^2[0,T]}}\bigg)\right)^2\right]<\infty.$
\end{proof}

\section{Strong convergence}\label{kja}
As discussed in the introduction, the strong convergence of the solutions to the equation (\ref{eq1}) have been well studied.  In this section, our aim is to prove the optimality of the estimate (\ref{yfla1}) as $\alpha\to \beta.$ More specifically, in Theorem \ref{dl32m} below, we show that the limit
$Y_{\beta,t}:=\lim\limits_{\alpha\to \beta}\frac{X_{\alpha,t} -X_{\beta,t}}{\alpha-\beta}$ exists and can be computed explicitly.

Let us introduce some notations: We denote by $C$ (with or without an index) a generic constant which may vary at each appearance. Given a function $h(t,x),$ we denote
$$h'(t,x):=\frac{\partial h}{\partial x}(t,x),\quad h''(t,x):=\frac{\partial h}{\partial x^2}(t,x).$$
For any $a,b\in \mathbb{R},$ we denote $a\vee b=\max\{a,b\}, \  a\wedge b=\min\{a,b\}.$ In the proofs, we frequently use the fundamental inequality
 \begin{equation}\label{pt1}
(a_1+\cdots+a_n)^p\leq n^{p-1}(a_1^p+\cdots+a_n^p) \ \ \mbox{for all}  \ a_1,...,a_n\geq 0 \ \  \mbox{and} \ \ p\geq 2.
\end{equation}
We also frequently use a trick of  the  H\"{o}lder inequality as follows: For all $ p>1, q>1$ with $\frac{1}{p}+\frac{1}{q}=1,$ we have
\begin{equation}\label{pt3}
\left(\int\limits_{a}^b|f(t)g(t)|dt\right)^p=\left(\int\limits_{a}^b|f(t)|^{\frac{1}{q}+\frac{1}{p}}|g(t)|dt\right)^p\leq \left(\int\limits_{a}^b|f(t)|dt\right)^{\frac{p}{q}}\int\limits_{a}^b|f(t)||g(t)|^pdt.
\end{equation}
For the existence and uniqueness of the solutions, we make the use of the following assumption, see e.g. \cite{SonDT,Wang1}.
\begin{assum}\label{asum1} The coefficients $b, \sigma  :[0,T]\times \mathbb{R}\to \mathbb{R}$ are Lipschitz and have linear growth, that is, there exists $L>0$ such that
 $$ |b(t,x)-b(t,y)| +|\sigma(t,x)-\sigma(t,y)|\le L|x-y|\,\forall x,y\in \mathbb{R},t\in[0,T],$$
 and
 $$ |b(t,x)|+|\sigma(t,x)|\le L(1+|x|)\, \forall x\in \mathbb{R},t\in[0,T].$$
\end{assum}
We first need the following technical results.
	\begin{prop}\label{pro1}Suppose Assumption \ref{asum1}. Let $(X_{\alpha,t})_{ t\in [0, T]}$ be the solution to the equation (\ref{eq1}). Then, for every $\frac{1}{2}<\alpha_0<1$ and $p \geq 2,$ we have
		\begin{align}\label{pt2}
		\sup\limits_{ \alpha\in \left[\alpha_0,1\right]}\sup\limits_{t\in [0,T]}E|X_{\alpha,t}|^p \le C,
		\end{align}
		where $C$ is a positive constant depending only on $L, T, p,$ $\alpha_0$ and $x_0.$
\end{prop}
\begin{proof} This proposition is not new. We give a proof to show that the moments are bounded uniformly in $\alpha.$ By using the inequality \eqref{pt1}, the fact $\Gamma(\alpha)\geq 1$ and Burkh\"older-Davis-Gundy inequality, it follows from the equation \eqref{eq1} that
\begin{align*}			
E&\left|X_{\alpha,t}\right|^p\\
&\leq 3^{p-1}\left(|x_0|^p+E\left|\frac{1}{\Gamma(\alpha)}\int_0^t(t-s)^{\alpha-1}b(s,X_{\alpha,s})ds\right|^p+E\left|\frac{1}{\Gamma(\alpha)}\int_0^t(t-s)^{\alpha-1}\sigma(s,X_{\alpha,s})dB_s\right|^p\right)\\
&\leq C+CE\left|\int_0^t(t-s)^{\alpha-1}b(s,X_{\alpha,s})ds\right|^p+CE\left|\int_0^t(t-s)^{2\alpha-2}\sigma^2(s,X_{\alpha,s})ds\right|^{\frac{p}{2}},\,\,0\leq t\leq T,
\end{align*}
where $C$ is a positive constant depending only on $ p$ and  $x_0.$
		
		Next, by using the  H\"{o}lder inequality (\ref{pt3}), we obtain
			\begin{align*}
			E&\left|X_{\alpha,t}\right|^p \\
&\le C+Ct^{\frac{p}{2}}E\left|\int_0^t(t-s)^{2\alpha-2}b^2(s,X_{\alpha,s})ds\right|^{\frac{p}{2}}+CE\left|\int_0^t(t-s)^{2\alpha-2}\sigma^2(s,X_{\alpha,s})ds\right|^{\frac{p}{2}}\\
			&\le C+Ct^{\frac{p}{2}}\left(\int_0^t(t-s)^{2\alpha-2}ds\right)^{\frac{p}{2}-1}\int_0^t(t-s)^{2\alpha-2}E|b(s,X_{\alpha,s})|^pds\\
			&\qquad+C\left(\int_0^t(t-s)^{2\alpha-2}ds\right)^{\frac{p}{2}-1}\int_0^t(t-s)^{2\alpha-2}E|\sigma(s,X_{\alpha,s})|^pds,\,\,0\leq t\leq T.
\end{align*}
Then, under Assumption \ref{asum1}, we deduce
\begin{align*}
			&E\left|X_{\alpha,t}\right|^p \leq C+2^{p}L^pCt^{\frac{p}{2}}\left(\frac{t^{2\alpha-1}}{2\alpha-1}\right)^{\frac{p}{2}-1}\int_0^t(t-s)^{2\alpha-2}(1+E|X_{\alpha,s}|^p)ds\\
&\leq C+C\int_0^t(t-s)^{2\alpha-2}E|X_{\alpha,s}|^pds,\  \ \ 0 \le t\le T,
			\end{align*}
				where $C$ is a positive constant depending only on $L, T, p,x_0$ and $\alpha_0.$ Applying the inequality (\ref{dhk}) to $\eta=2\alpha-1,$ we get
$$E\left|X_{\alpha,t}\right|^p \leq 2C\exp\left(\frac{4^{\frac{2}{\eta}}C^{\frac{2}{\eta}}t^2}{\eta^{\frac{2}{\eta}}}\right),\,\,0 \le t\le T,$$
 which give us the required conclusion \eqref{pt2}. The proof of the proposition is complete.
		\end{proof}
 \begin{prop}\label{dl32} Suppose Assumption \ref{asum1}. Let $(X_{\alpha,t})_{t\in [0,T]}$ and $(X_{\beta,t})_{t\in [0,T]}$ be the solutions to the equations (\ref{eq1}) and  (\ref{eq1qc}), respectively. Then, for  every $p\geq 2$ and $\alpha_0\in (\frac{1}{2}, 1],$ we have
		\begin{align}\label{eq2}
E|	X_{\alpha,t} -X_{\beta,t}|^p \le Ct^{p(\alpha\wedge\beta)-\frac{p}{2}}(|\ln t|^p+1)|\alpha-\beta|^p\,\,\forall\,t\in [0,T],\alpha,\beta\in [\alpha_0,1],
		\end{align}
		where $C$ is a positive constant depending only on $L, T, p,$ $\alpha_0$ and $x_0$.
	\end{prop}
\begin{proof}
We recall that, for $0\leq t\leq T,$
$$X_{\alpha,t} = x_{0} + \frac{1}{\Gamma(\alpha)} \left(\int_0^t (t-s)^{\alpha -1}b(s,X_{\alpha,s})ds +  \int_0^t (t-s)^{\alpha -1} \sigma(s,X_{\alpha,s})dB_s\right)$$
and
$$X_{\beta,t} = x_{0} + \frac{1}{\Gamma(\beta)} \left(\int_0^t (t-s)^{\beta -1}b(s,X_{\beta,s})ds +  \int_0^t (t-s)^{\beta -1} \sigma(s,X_{\beta,s})dB_s\right).$$
We therefore get
		\begin{align}
		X_{\alpha,t} -X_{\beta,t}&= \frac{1}{\Gamma(\alpha)}\int_0^t (t-s)^{\alpha -1}b(s,X_{\alpha,s})ds 	- \frac{1}{\Gamma(\beta)}\int_0^t (t-s)^{\beta -1}b(s,X_{\beta,s})ds\notag\\
		& + \frac{1}{\Gamma(\alpha)} \int_0^t (t-s)^{\alpha -1} \sigma(s,X_{\alpha,s})dB_s
		- \frac{1}{\Gamma(\beta)}\int_0^t (t-s)^{\beta -1} \sigma(s,X_{\beta,s})dB_s\notag\\
		& = I_1+I_2,\,\,0\leq t\leq T,\label{almv}
		\end{align}
		where
		\begin{align*}
		I_1:= \frac{1}{\Gamma(\alpha)}\int_0^t (t-s)^{\alpha -1}b(s,X_{\alpha,s})ds 	- \frac{1}{\Gamma(\beta)}\int_0^t (t-s)^{\beta -1}b(s,X_{\beta,s})ds,
		\end{align*}
		and \begin{align*}
		I_2:= \frac{1}{\Gamma(\alpha)} \int_0^t (t-s)^{\alpha -1} \sigma(s,X_{\alpha,s})dB_s
		- \frac{1}{\Gamma(\beta)}\int_0^t (t-s)^{\beta -1} \sigma(s,X_{\beta,s})dB_s.
		\end{align*}
To estimate $E|I_1|^p$ and $E|I_2|^p$ we use the following decompositions
		\begin{align*}
		I_1 & =  \frac{1}{\Gamma(\alpha)} \int_0^t (t-s)^{\alpha -1} \big(b(s,X_{\alpha,s}) - b(s,X_{\beta,s}) \big) ds\\
&+  \frac{1}{\Gamma(\alpha)}\int_0^t   ((t-s)^{\alpha -1} - (t-s)^{\beta -1}) b(s,X_{\beta,s})ds\\
		& + \bigg(\frac{1}{\Gamma(\alpha)} -\frac{1}{\Gamma(\beta)}\bigg)\int_0^t (t-s)^{\beta -1} b(s,X_{\beta,s})ds.	
		\end{align*}
		and
		\begin{align*}
		I_2 &= \frac{1}{\Gamma(\alpha)} \int_0^t (t-s)^{\alpha -1} \big(\sigma(s,X_{\alpha,s}) - \sigma(s,X_{\beta,s}) \big) dB_s\\
		&+ \frac{1}{\Gamma(\alpha)}\int_0^t   ((t-s)^{\alpha -1} - (t-s)^{\beta -1}) \sigma(s,X_{\beta,s})dB_s\\
		& + \bigg(\frac{1}{\Gamma(\alpha)} -\frac{1}{\Gamma(\beta)}\bigg)\int_0^t (t-s)^{\beta -1} \sigma(s,X_{\beta,s})dB_s.
		\end{align*}
		By the inequality \eqref{pt1}, the Cauchy-Schwarz inequality and Assumption \ref{asum1}, we deduce
		\begin{align*}
		E|I_1|^p &\le 3^{p-1}L^pt^{\frac{p}{2}} E\left(\int_0^t(t-s)^{2\alpha-2} |X_{\alpha,s} - X_{\beta,s}|^2ds \right)^{\frac{p}{2}}\\&
		+ 6^{p-1}L^pt^{\frac{p}{2}} E\left(\int_0^t  \left((t-s)^{\alpha -1} - (t-s)^{\beta -1}\right)^2(1+|X_{\beta,s}|)^2ds\right)^{\frac{p}{2}}\\
		& +  6^{p-1}L^pt^{\frac{p}{2}}\bigg|\frac{1}{\Gamma(\alpha)} -\frac{1}{\Gamma(\beta)}\bigg|^pE\left(\int_0^t (t-s)^{2\beta -2} (1+|X_{\beta,s}|^2)ds\right)^{\frac{p}{2}}.
\end{align*}
Then, by the inequality (\ref{pt3}), we get
\begin{align*}
	&	E|I_1|^p		 \le 3^{p-1}L^pt^{\frac{p}{2}}\left(\int_0^t(t-s)^{2\alpha-2}ds\right)^{\frac{p}{2}-1} \int_0^t(t-s)^{2\alpha-2} E|X_{\alpha,s} - X_{\beta,s}|^pds \\
		& + 6^{p-1}L^p(2t)^{\frac{p}{2}}\left(\int_0^t\left((t-s)^{\alpha -1} - (t-s)^{\beta -1}\right)^2ds\right)^{\frac{p}{2}-1}\\
&\times\int_0^t  \left((t-s)^{\alpha -1} - (t-s)^{\beta -1}\right)^2(1+E|X_{\beta,s}|^p)ds\\
		&+6^{p-1}L^pt^{\frac{p}{2}}\bigg|\frac{1}{\Gamma(\alpha)} -\frac{1}{\Gamma(\beta)}\bigg|^p\left(\int_0^t(t-s)^{2\beta-2}ds\right)^{\frac{p}{2}-1} \int_0^t(t-s)^{2\beta-2} (1+E|X_{\beta,s}|^p)ds.
\end{align*}
We now use the estimate (\ref{pt2})  to get
\begin{align*}
		E|I_1|^p&		 \le C \int_0^t(t-s)^{2\alpha-2} E|X_{\alpha,s} - X_{\beta,s}|^pds \\
		& + Ct^{\frac{p}{2}}\left(\int_0^t\left((t-s)^{\alpha -1} - (t-s)^{\beta -1}\right)^2ds\right)^{\frac{p}{2}}+Ct^{p\beta}\bigg|\frac{1}{\Gamma(\alpha)} -\frac{1}{\Gamma(\beta)}\bigg|^p,
				\end{align*}
where $C$ is a positive constant depending only on $L,T,p,x_0$ and $\alpha_0.$ Recalling Lemmas \ref{l2} and \ref{lm3}, we obtain
\begin{equation}\label{alm}
				E|I_1|^p
				\le  C(t^{p\alpha}+t^{p\beta})(|\ln t|^p+1)|\alpha-\beta|^p  + C \int_0^t(t-s)^{2\alpha-2} E|X_{\alpha,s}- X_{\beta,s}|^pds,
		\end{equation}
		where $C$ is a positive constant depending  on $L, T, p,$ $\alpha_0$ and $x_0$.
	
	For 	$E|I_2|^p$, 	using  the inequality \eqref{pt1}, Burkh\"older-Davis-Gundy inequality,   Assumption \ref{asum1}, we have
		\begin{align*}
		E|I_2|^p &\leq 3^{p-1}L^pc_pE\left(\int_0^t(t-s)^{2\alpha-2} |X_{\alpha,s} - X_{\beta,s}|^2ds \right)^{\frac{p}{2}}\\&
		+ 6^{p-1}L^pc_p E\left(\int_0^t  \left((t-s)^{\alpha -1} - (t-s)^{\beta -1}\right)^2(1+|X_{\beta,s}|)^2ds\right)^{\frac{p}{2}}\\
		& +  3^{p-1}L^pc_p\bigg|\frac{1}{\Gamma(\alpha)} -\frac{1}{\Gamma(\beta)}\bigg|^pE\left(\int_0^t (t-s)^{2\beta-2} (1+|X_{\beta,s}|)^2ds\right)^{\frac{p}{2}},
		\end{align*}
			where $c_p$ is the constant in the  Burkh\"older-Davis-Gundy inequality. Hence, by using the same arguments as in the proof of (\ref{alm}), we obtain
\begin{equation}\label{alma}
E|I_2|^p
\leq C(t^{p\alpha-p/2}+t^{p\beta-p/2})(|\ln t|^p+1)|\alpha-\beta|^p  + C \int_0^t(t-s)^{2\alpha-1} E|X_{\alpha,s}- X_{\beta,s}|^pds,
\end{equation}
	 where   $C$ is a positive constant depending  only on $L,T, p,$ $\alpha_0$ and $x_0$. Combining (\ref{almv}), (\ref{alm}) and (\ref{alma}) yields
		\begin{align}
		&E|X_{\alpha,t} -X_{\beta,t}|^p \le 2^{p-1}(E|I_1|^p+ E|I_2|^p)\notag\\
		& \le C t^{p(\alpha\wedge\beta)-\frac{p}{2}}(|\ln t|^p+1)|\alpha -\beta|^p  + C \int_0^t (t-s)^{2\alpha-2}E|X_{\alpha,s} - X_{\beta,s}|^p ds,\,\,0\leq t\leq T.\label{caik}
		\end{align}
As a consequence, applying the inequality (\ref{dhk3}) to $\eta=2\alpha-1$  gives us
\begin{align*}
&E|X_{\alpha,t} -X_{\beta,t}|^p\\
& \le C t^{p(\alpha\wedge\beta)-\frac{p}{2}}(|\ln t|^p+1)|\alpha -\beta|^p  + C|\alpha -\beta|^p\left(\int_0^t  \big(s^{p(\alpha\wedge\beta)-\frac{p}{2}}(|\ln s|^p+1)\big)^{\frac{2}{\eta}}ds\right)^{\frac{\eta}{2}}\\
&=C t^{p(\alpha\wedge\beta)-\frac{p}{2}}(|\ln t|^p+1)|\alpha -\beta|^p  +Ct^{p(\alpha\wedge\beta)-\frac{p}{2}}|\alpha -\beta|^p\left(\int_0^1 \big(s^{p(\alpha\wedge\beta)-\frac{p}{2}}(|\ln t+\ln s|^p+1)\big)^{\frac{2}{\eta}}ds\right)^{\frac{\eta}{2}}\\
&\leq C t^{p(\alpha\wedge\beta)-\frac{p}{2}}(|\ln t|^p+1)|\alpha -\beta|^p  +Ct^{p(\alpha\wedge\beta)-\frac{p}{2}}|\alpha -\beta|^p\left(\int_0^1 \big(s^{p(\alpha\wedge\beta)-\frac{p}{2}}(|\ln t|^p+|\ln s|^p+1)\big)^{\frac{2}{\eta}}ds\right)^{\frac{\eta}{2}}.
\end{align*}	
Furthermore, it is easy to check that the integral $\left(\int_0^1 \big(s^{p(\alpha\wedge\beta)-\frac{p}{2}}|\ln s|^p\big)^{\frac{2}{\eta}}ds\right)^{\frac{\eta}{2}}$ is bounded uniformly in $\alpha,\beta\in[\alpha_0,1].$ So we conclude that
$$
E|X_{\alpha,t} -X_{\beta,t}|^p \le  C t^{p(\alpha\wedge\beta)-\frac{p}{2}}(|\ln t|^p+1)|\alpha -\beta|^p,\,\,0\leq t\leq T.$$
The proof of the proposition is complete.
\end{proof}

The estimate \eqref{eq2} suggests that the limit $\lim\limits_{\alpha\to \beta}\frac{X_{\alpha,t} -X_{\beta,t}}{\alpha-\beta}$ exists. Note that a nonzero limit ensures the optimality of the estimate \eqref{eq2} for the strong convergence.  To end this, let us additionally assume that
\begin{assum}\label{asum2}
For each $t\in[0,T],$ the functions $b(t,.),\sigma(t,.)$ are differentiable on $\mathbb{R}$ and satisfy
$$ |b'(t,x)-b'(t,y)| +|\sigma'(t,x)-\sigma'(t,y)|\le L(1+|x|^\nu+|y|^\nu)|x-y|^\delta\,\forall x,y\in \mathbb{R},t\in[0,T]$$
for some $L,\nu>0$ and $\delta\in (0,1].$
\end{assum}
In the next theorem, we are able to find the exact value for the limit of $\frac{X_{\alpha,t} -X_{\beta,t}}{\alpha-\beta}$ as $\alpha\to \beta$ and the distance from $\frac{X_{\alpha,t} -X_{\beta,t}}{\alpha-\beta}$ to its limit $Y_{\beta,t}.$
\begin{theorem}\label{dl32m}
	Suppose Assumptions \ref{asum1} and \ref{asum2}. Let $(X_{\alpha,t})_{t\in [0,T]}$ and $(X_{\beta,t})_{t\in [0,T]}$ be the solutions to the equations (\ref{eq1}) and  (\ref{eq1qc}), respectively. Then, for  every $p\geq 2$ and $\alpha_0\in (\frac{1}{2}, 1],$ we have
	\begin{align}\label{eq2m}
 \sup\limits_{0\le t\le T}E\left|\frac{X_{\alpha,t} -X_{\beta,t}}{\alpha-\beta}-Y_{\beta,t}\right|^p \le C|\alpha-\beta|^{p\delta}\,\,\forall\,\alpha,\beta\in [\alpha_0,1],
	\end{align}
	where $C$ is a positive constant depending only on $L, T, p,$ $\alpha_0$, $x_0,\nu,\delta$ and $(Y_{\beta,t})_{0\leq t\leq T}$ is defined by
		\begin{align}\label{eq1m}
	Y_{\beta,t} &= -\frac{\Gamma'(\beta)}{\Gamma^2(\beta)} \left(\int_0^t (t-s)^{\beta -1}b(s,X_{\beta,s})ds +  \int_0^t (t-s)^{\beta -1} \sigma(s,X_{\beta,s})dB_s\right)\notag \\&
\ \ +	\frac{1}{\Gamma(\beta)} \int_0^t (t-s)^{\beta -1}\left(\ln(t-s)b(s,X_{\beta,s})+b'(s,X_{\beta,s})Y_{\beta,t}\right)ds\notag\\&
		 \ \ +\frac{1}{\Gamma(\beta)} \int_0^t (t-s)^{\beta -1}\left(\ln(t-s)\sigma(s,X_{\beta,s})+\sigma'(s,X_{\beta,s})Y_{\beta,t}\right)dB_s,\,\,0\leq t\leq T.
	\end{align}
	\end{theorem}
\begin{proof}
For $\alpha\neq \beta,$ we obtain from the relation (\ref{almv}) that
	\begin{align*}
	&\frac{X_{\alpha,t} -X_{\beta,t}}{\alpha-\beta}
= \frac{1}{\alpha-\beta}\left(\frac{1}{\Gamma(\alpha)}\int_0^t (t-s)^{\alpha -1}b(s,X_{\alpha,s})ds - \frac{1}{\Gamma(\beta)}\int_0^t (t-s)^{\beta -1}b(s,X_{\beta,s})ds\right)\\
&	 +   \frac{1}{\alpha-\beta}\left(\frac{1}{\Gamma(\alpha)}\int_0^t (t-s)^{\alpha -1}\sigma(s,X_{\alpha,s})dB_s- \frac{1}{\Gamma(\beta)}\int_0^t (t-s)^{\beta -1}\sigma(s,X_{\beta,s})dB_s\right),\,0\leq t\leq T.
	\end{align*}
We decompose the addends in the right hand side as follows
		\begin{align}
	&	\frac{1}{\alpha-\beta}\left(\frac{1}{\Gamma(\alpha)}\int_0^t (t-s)^{\alpha -1}b(s,X_{\alpha,s})ds-\frac{1}{\Gamma(\beta)}\int_0^t (t-s)^{\beta -1}b(s,X_{\beta,s})ds\right)\notag\\
	&= \frac{1}{\alpha-\beta}\left(\frac{1}{\Gamma(\alpha)}-\frac{1}{\Gamma(\beta)}\right)\int_0^t (t-s)^{\alpha -1}b(s,X_{\alpha,s})ds 	+ \frac{1}{\Gamma(\beta)}\int_0^t \frac{(t-s)^{\alpha -1}-(t-s)^{\beta -1}}{\alpha-\beta}b(s,X_{\beta,s})ds\notag\\
	& \ \  +  \frac{1}{\Gamma(\beta)}\int_0^t\left((t-s)^{\alpha -1}-(t-s)^{\beta-1}\right) \frac{b(s,X_{\alpha,s})-b(s,X_{\beta,s})}{\alpha-\beta}ds\notag\\& \ \ +  \frac{1}{\Gamma(\beta)}\int_0^t(t-s)^{\beta -1} \frac{b(s,X_{\alpha,s})-b(s,X_{\beta,s})}{\alpha-\beta}ds,\label{9m1}
	\end{align}
	and
		\begin{align}
	&	\frac{1}{\alpha-\beta}\left(\frac{1}{\Gamma(\alpha)}\int_0^t (t-s)^{\alpha -1}\sigma(s,X_{\alpha,s})dB_s- \frac{1}{\Gamma(\beta)}\int_0^t (t-s)^{\beta -1}\sigma(s,X_{\beta,s})dB_s\right)\notag\\&=  \frac{1}{\alpha-\beta}\left(\frac{1}{\Gamma(\alpha)}-\frac{1}{\Gamma(\beta)}\right)\int_0^t (t-s)^{\alpha -1}\sigma(s,X_{\alpha,s})dB_s \notag\\&	+ \frac{1}{\Gamma(\beta)}\int_0^t \frac{(t-s)^{\alpha -1}-(t-s)^{\beta -1}}{\alpha-\beta}\sigma(s,X_{\beta,s})dB_s\notag\\
	& \ \  +  \frac{1}{\Gamma(\beta)}\int_0^t\left((t-s)^{\alpha -1}-(t-s)^{\beta-1}\right) \frac{\sigma(s,X_{\alpha,s})-\sigma(s,X_{\beta,s})}{\alpha-\beta}dB_s\notag\\& \ \ +  \frac{1}{\Gamma(\beta)}\int_0^t(t-s)^{\beta -1} \frac{\sigma(s,X_{\alpha,s})-\sigma(s,X_{\beta,s})}{\alpha-\beta}dB_s.\label{9m2}
	\end{align}
By Taylor's expansion, we have
\begin{equation}\label{9i1}
\frac{1}{\Gamma(\alpha)}-\frac{1}{\Gamma(\beta)}
=-(\alpha-\beta)\frac{\Gamma'(\beta)}{\Gamma^2(\beta)}+(\alpha-\beta)R_1(\alpha,\beta),
\end{equation}
where the remainder $R_1(\alpha,\beta)$ is given by
$$R_1(\alpha,\beta):=\frac{1}{2}\left(-\frac{\Gamma''(\beta+\theta(\alpha-\beta))}{\Gamma^2(\beta+\theta(\alpha-\beta))}
+\frac{\Gamma'(\beta+\theta(\alpha-\beta))^2}{\Gamma^3(\beta+\theta(\alpha-\beta))}\right)(\alpha-\beta)$$
for some $\theta\in(0,1).$ Furthermore, by the part $(i)$ of Lemma \ref{l1} and the fact $\Gamma(\beta+\theta(\alpha-\beta))\geq 1,$ this remainder satisfies
\begin{equation}\label{pt1.1}
	|R_1(\alpha,\beta)|\leq C|\alpha-\beta|,
	\end{equation}
where $C>0$ is an absolute constant.

For $0\leq s\leq t\leq T,$ we have
		\begin{align}
(t-s)^{\alpha -1}-(t-s)^{\beta -1}&=(\alpha-\beta)\int_0^1 (t-s)^{\beta-\theta(\alpha-\beta)}\ln(t-s)d\theta\notag\\
&=(\alpha-\beta)(t-s)^{\beta}\ln(t-s)+(\alpha-\beta)R_2(\alpha,\beta,t,s),\label{9i2}
	\end{align}
		where the remainder term $R_2(\alpha,\beta,t,s)$ is defined by
	$$R_2(\alpha,\beta,t,s):=\int_0^1\left( (t-s)^{\beta-\theta(\alpha-\beta)}-(t-s)^{\beta}\right)\ln(t-s)d\theta.$$
By using the same arguments as in the proof of (\ref{hh2.1}), this remainder satisfies
		\begin{align}
	\int_0^t|R_2(\alpha,\beta,t,s)|^2ds&\leq \int_0^t\int_0^1\left( (t-s)^{\beta-\theta(\alpha-\beta)}-(t-s)^{\beta}\right)^2\ln^2(t-s)d\theta ds\notag \\&
	=\int_0^1\left(\int_0^t\left( (t-s)^{\beta-\theta(\alpha-\beta)}-(t-s)^{\beta}\right)^2\ln^2(t-s)ds\right)d\theta\notag\\&
	\leq C(\alpha-\beta)^2,\label{pt1.2}
	\end{align}
where $C$ is a positive constant depending only on $T$ and $\alpha_0.$

For every $s\in [0,T],$ we have
		\begin{align}
b(s,X_{\alpha,s})-b(s,X_{\beta,s})&=(X_{\alpha,s}-X_{\beta,s})\int_0^1 b'(s,X_{\beta,s}+\theta(X_{\alpha,s}-X_{\beta,s}))d\theta\notag\\
&=b'(s,X_{\beta,s})(X_{\alpha,s}-X_{\beta,s})+R_b(\alpha,\beta, t,s),\label{9i3}
	\end{align}
	where the term $R_b(\alpha,\beta, t,s)$ is defined by
	$$R_b(\alpha,\beta, t,s):=(X_{\alpha,s}-X_{\beta,s})\int_0^1\left( b'(s,X_{\beta,s}+\theta(X_{\alpha,s}-X_{\beta,s}))-b'(s,X_{\beta,s})\right)d\theta$$
By Assumption \ref{asum2} we have
	\begin{align*}
	|R_b(\alpha,\beta, t,s)|&=|X_{\alpha,s}-X_{\beta,s}|\int_0^1L(1+|X_{\beta,s}|^\nu+|\theta(X_{\alpha,s}-X_{\beta,s})|^\nu)|\theta(X_{\alpha,s}-X_{\beta,s})|^{\delta} d\theta\\&\leq C(1+|X_{\alpha,s}|^\nu+|X_{\beta,s}|^\nu)|X_{\alpha,s}-X_{\beta,s}|^{1+\delta}, \ \ 0\leq s\leq t\leq T.
	\end{align*}
where $C$ is a  positive constant depending only on $L$ and $\nu.$ Hence, by Cauchy-Schwarz inequality and the  inequality \eqref{pt1}, we get
 	\begin{align*}
 E|R_b(\alpha,\beta, t,s)|^p&=\sqrt{3^{p-1}C(1+E|X_{\alpha,s}|^{2p\nu}+E|X_{\beta,s}|^{2p\nu})E|X_{\alpha,s}-X_{\beta,s}|^{2p(1+\delta)}}.
 \end{align*}
 This, together with the estimates obtained in and Propositions \ref{pro1} and \ref{dl32}, yields
 \begin{equation}\label{pt1.3}	
 E|R_b(\alpha,\beta, t,s)|^p\leq C|\alpha-\beta|^{p(1+\delta)}.
 \end{equation}
 Similarly, we have
\begin{equation}\label{9i4}
 \sigma(s,X_{\alpha,s})-\sigma(s,X_{\beta,s})
 =\sigma'(s,X_{\beta,s})(X_{\alpha,s}-X_{\beta,s})+R_\sigma(\alpha,\beta, t,s),
\end{equation}
 where the remainder term $R_\sigma(\alpha,\beta, t,s)$ satisfies
 \begin{equation}\label{pt1.4}	
 E|R_\sigma(\alpha,\beta, t,s)|^p\leq C|\alpha-\beta|^{p(1+\delta)}.
 \end{equation}
 We insert (\ref{9i1}), (\ref{9i2}), (\ref{9i3}) and (\ref{9i4}) into the decompositions (\ref{9m1}) and (\ref{9m2}) to get
 	\begin{align*}
	\frac{X_{\alpha,t} -X_{\beta,t}}{\alpha-\beta}&= \left(-\frac{\Gamma'(\beta)}{\Gamma^2(\beta)}+R_1(\alpha,\beta)\right)\int_0^t (t-s)^{\alpha -1}b(s,X_{\alpha,s})ds \\&	+ \frac{1}{\Gamma(\beta)}\int_0^t \left((t-s)^{\beta}\ln(t-s)+R_2(\alpha,\beta,t,s)\right)b(s,X_{\beta,s})ds\\
 &  +  \frac{1}{\Gamma(\beta)}\int_0^t\left((t-s)^{\alpha -1}-(t-s)^{\beta-1}\right) \frac{b(s,X_{\alpha,s})-b(s,X_{\beta,s})}{\alpha-\beta}ds\\
 & +  \frac{1}{\Gamma(\beta)}\int_0^t(t-s)^{\beta -1} \left(b'(s,X_{\beta,s})\frac{X_{\alpha,s}-X_{\beta,s}}{\alpha-\beta}+\frac{R_b(\alpha,\beta, t,s)}{\alpha-\beta}\right)ds\\& + \left(-\frac{\Gamma'(\beta)}{\Gamma^2(\beta)}+R_1(\alpha,\beta)\right)\int_0^t (t-s)^{\alpha -1}\sigma(s,X_{\alpha,s})dB_s \\
 & 	+ \frac{1}{\Gamma(\beta)}\int_0^t \left((t-s)^{\beta}\ln(t-s)+R_2(\alpha,\beta,t,s)\right)\sigma(s,X_{\beta,s})dB_s\\
 &  +  \frac{1}{\Gamma(\beta)}\int_0^t\left((t-s)^{\alpha -1}-(t-s)^{\beta-1}\right) \frac{\sigma(s,X_{\alpha,s})-\sigma(s,X_{\beta,s})}{\alpha-\beta}dB_s\\
 & +  \frac{1}{\Gamma(\beta)}\int_0^t(t-s)^{\beta -1} \left(\sigma'(s,X_{\beta,s})\frac{X_{\alpha,s}-X_{\beta,s}}{\alpha-\beta}+\frac{R_\sigma(\alpha,\beta, t,s)}{\alpha-\beta}\right)dB_s,\,\,0\leq t\leq T.
 \end{align*}
 We recall that
 \begin{align}
	Y_{\beta,t} &= -\frac{\Gamma'(\beta)}{\Gamma^2(\beta)} \left(\int_0^t (t-s)^{\beta -1}b(s,X_{\beta,s})ds +  \int_0^t (t-s)^{\beta -1} \sigma(s,X_{\beta,s})dB_s\right)\notag \\&
\ \ +	\frac{1}{\Gamma(\beta)} \int_0^t (t-s)^{\beta -1}\left(\ln(t-s)b(s,X_{\beta,s})+b'(s,X_{\beta,s})Y_{\beta,t}\right)ds\notag\\&
		 \ \ +\frac{1}{\Gamma(\beta)} \int_0^t (t-s)^{\beta -1}\left(\ln(t-s)\sigma(s,X_{\beta,s})+\sigma'(s,X_{\beta,s})Y_{\beta,t}\right)dB_s,\,\,0\leq t\leq T.\notag
	\end{align}
 Then, for $Z_{\alpha,\beta,t}:=\frac{X_{\alpha,t} -X_{\beta,t}}{\alpha-\beta}-Y_{\beta,t},$ we obtain
 \begin{align}
 Z_{\alpha,\beta,t}&=\frac{1}{\Gamma(\beta)}\left(\int_0^t(t-s)^{\beta -1}b'(s,X_{\beta,s}) Z_{\alpha,\beta,s}ds+\int_0^t(t-s)^{\beta -1}\sigma'(s,X_{\beta,s}) Z_{\alpha,\beta,s}dB_s\right)\notag\\
 &+R(\alpha,\beta,t)+\bar{R}(\alpha,\beta,t),\,\,0\leq t\leq T,\label{skm}
 \end{align}
 where the terms $R(\alpha,\beta,t)$ and $\bar{R}(\alpha,\beta,t)$ are defined by
\begin{multline}
R(\alpha,\beta,t):=R_1(\alpha,\beta)\left(\int_0^t (t-s)^{\alpha -1}b(s,X_{\alpha,s})ds+\int_0^t (t-s)^{\alpha -1}\sigma(s,X_{\alpha,s})dB_s\right)\\
+\frac{1}{\Gamma(\beta)}\int_0^t R_2(\alpha,\beta,t,s)b(s,X_{\beta,s})ds+\frac{1}{\Gamma(\beta)}\int_0^t R_2(\alpha,\beta,t,s)\sigma(s,X_{\beta,s})dB_s\\
+\frac{1}{\Gamma(\beta)}\int_0^t(t-s)^{\beta -1} \frac{R_b(\alpha,\beta, t,s)}{\alpha-\beta}ds+\frac{1}{\Gamma(\beta)}\int_0^t(t-s)^{\beta -1} \frac{R_\sigma(\alpha,\beta, t,s)}{\alpha-\beta}dB_s
\end{multline}
and
\begin{multline}
\bar{R}(\alpha,\beta,t)
:=-\frac{\Gamma'(\beta)}{\Gamma^2(\beta)}\int_0^t ((t-s)^{\alpha -1}b(s,X_{\alpha,s})+(t-s)^{\alpha -1}\sigma(s,X_{\alpha,s}))ds\\
+\frac{\Gamma'(\beta)}{\Gamma^2(\beta)}\int_0^t ((t-s)^{\beta -1}b(s,X_{\beta,s})+(t-s)^{\beta -1}\sigma(s,X_{\beta,s}))dB_s\\
+  \frac{1}{\Gamma(\beta)}\int_0^t\left((t-s)^{\alpha -1}-(t-s)^{\beta-1}\right) \frac{b(s,X_{\alpha,s})-b(s,X_{\beta,s})}{\alpha-\beta}ds\\
+\frac{1}{\Gamma(\beta)}\int_0^t\left((t-s)^{\alpha -1}-(t-s)^{\beta-1}\right) \frac{\sigma(s,X_{\alpha,s})-\sigma(s,X_{\beta,s})}{\alpha-\beta}dB_s.
\end{multline}
For every $p\geq 2,$ by using the Cauchy-Schwarz and Burkh\"older-Davis-Gundy inequalities and the fact $\Gamma(\beta)\geq 1,$ we obtain
\begin{multline*}
E|R(\alpha,\beta,t)|^p\leq C|R_1(\alpha,\beta)|^pE|\Gamma(\alpha)(X_{\alpha,t}-x_0)|^p\\
+Ct^{\frac{p}{2}}E\left(\int_0^t |R_2(\alpha,\beta,t,s)b(s,X_{\beta,s})|^2ds\right)^{\frac{p}{2}}+CE\left(\int_0^t |R_2(\alpha,\beta,t,s)\sigma(s,X_{\beta,s})|^2ds\right)^{\frac{p}{2}}\\
+Ct^{\frac{p}{2}}E\left(\int_0^t(t-s)^{2\beta -2} \frac{|R_b(\alpha,\beta, t,s)|^2}{(\alpha-\beta)^2}ds\right)^{\frac{p}{2}}+CE\left(\int_0^t(t-s)^{2\beta -2} \frac{|R_\sigma(\alpha,\beta, t,s)|^2}{(\alpha-\beta)^2}ds\right)^{\frac{p}{2}},
\end{multline*}
where $C$ is a positive constant depending only on $p.$ Then, by the inequality (\ref{pt3}), we deduce
\begin{multline*}
E|R(\alpha,\beta,t)|^p\leq C|R_1(\alpha,\beta)|^pE|\Gamma(1/2)(X_{\alpha,t}-x_0)|^p\\
+C\left(\int_0^t |R_2(\alpha,\beta,t,s)|^2ds\right)^{\frac{p}{2}-1}\int_0^t |R_2(\alpha,\beta,t,s)|^2\left(E|b(s,X_{\beta,s})|^{p}+E|\sigma(s,X_{\beta,s})|^{p}\right)ds\\
+CE\left(\int_0^t(t-s)^{2\beta -2}ds\right)^{\frac{p}{2}-1}\int_0^t(t-s)^{2\beta -2} \frac{E|R_b(\alpha,\beta, t,s)|^p+E|R_\sigma(\alpha,\beta, t,s)|^p}{|\alpha-\beta|^p}ds,
\end{multline*}
where $C$ is a positive constant depending only on $T$ and $p.$ Consequently, by using the estimate (\ref{pt2}), the linear growth property of $b$ and $\sigma$ and the estimates (\ref{pt1.1}), (\ref{pt1.2}), (\ref{pt1.3}) and (\ref{pt1.4}), we get
\begin{equation}
E|R(\alpha,\beta,t)|^p\leq C|\alpha-\beta|^{p\delta},
\end{equation}
where $C$ is a positive constant depending only on $L, T, p,$ $\alpha_0$, $x_0,\nu,\delta.$

Similarly, we have the following estimates for the term $\bar{R}(\alpha,\beta,t)$
\begin{multline*}
E|\bar{R}(\alpha,\beta,t)|^p\leq CE|\Gamma(\alpha)(X_{\alpha,t}-x_0)-\Gamma(\beta)(X_{\beta,t}-x_0)|^p\\
+  Ct^{\frac{p}{2}}E\left(\int_0^t\left((t-s)^{\alpha -1}-(t-s)^{\beta-1}\right)^2 \frac{|b(s,X_{\alpha,s})-b(s,X_{\beta,s})|^2}{(\alpha-\beta)^2}ds\right)^{\frac{p}{2}}\\
+CE\left(\int_0^t\left((t-s)^{\alpha -1}-(t-s)^{\beta-1}\right)^2 \frac{|\sigma(s,X_{\alpha,s})-\sigma(s,X_{\beta,s})|^2}{(\alpha-\beta)^2}ds\right)^{\frac{p}{2}},
\end{multline*}
and hence,
\begin{multline*}
E|\bar{R}(\alpha,\beta,t)|^p\leq C\Gamma^p(\alpha)\Gamma^p(\beta)\left|\frac{1}{\Gamma(\alpha)}-\frac{1}{\Gamma(\alpha)}\right|^pE|X_{\alpha,t}-x_0|^p+\Gamma^p(\beta)E|X_{\alpha,t}-X_{\beta,t}|^p\\
+  CE\left(\int_0^t\left((t-s)^{\alpha -1}-(t-s)^{\beta-1}\right)^2ds\right)^{\frac{p}{2}-1}\\
\times \int_0^t\left((t-s)^{\alpha -1}-(t-s)^{\beta-1}\right)^2 \frac{E|b(s,X_{\alpha,s})-b(s,X_{\beta,s})|^p+E|\sigma(s,X_{\alpha,s})-\sigma(s,X_{\beta,s})|^p}{|\alpha-\beta|^p}ds.
\end{multline*}
Note that $\Gamma(\alpha),\Gamma(\beta)\leq \Gamma(1/2).$ So it follows from Lipschitz property of $b$ and $\sigma$ and the estimates (\ref{tfg}), (\ref{hh2}), (\ref{eq2}) and (\ref{pt2}) that
\begin{equation}
E|\bar{R}(\alpha,\beta,t)|^p\leq C|\alpha-\beta|^p\leq C|\alpha-\beta|^{p\delta},
\end{equation}
where $C$ is a positive constant depending only on $L, T, p,$ $\alpha_0$, $x_0.$

To finish the proof, we use the same arguments as in the proof of Proposition \ref{pro1} and we obtain from (\ref{skm}) that, for $p\geq 2,$
\begin{align*}
E|Z_{\alpha,\beta,t}|^p&
\leq C\int_0^t(t-s)^{2\beta-2}E|Z_{\alpha,\beta,s}|^pds+CE|R(\alpha,\beta,t)|^p+E|\bar{R}(\alpha,\beta,t)|^p\\
&\leq C|\alpha-\beta|^{p\delta}+C\int_0^t(t-s)^{2\beta-2}E|Z_{\alpha,\beta,s}|^pds,\,\,0\leq t\leq T,
\end{align*}
where $C$ is a positive constant depending only on $L, T, p,\alpha_0, x_0,\nu$ and $\delta.$ Then, by applying the inequality (\ref{dhk}) to $\eta=2\beta-1,$ we conclude that
$$E|Z_{\alpha,\beta,t}|^p\leq C|\alpha-\beta|^{p\delta},\  \ \ 0 \le t\le T.$$
This completes the proof of the theorem because $Z_{\alpha,\beta,t}=\frac{X_{\alpha,t} -X_{\beta,t}}{\alpha-\beta}-Y_{\beta,t}.$
\end{proof}
\section{Weak convergence}\label{kjb}
In this section, our aim is to bound the quantity $|Eg(X_{\alpha,t}) -Eg(X_{\beta,t})|$ when $g$ is a bounded and measurable function. We are going to impose the following assumption.
\begin{assum}\label{asum3}For each $t\in[0,T],$ the coefficients $b(t,.),\sigma(t,.)$ are twice differentiable functions and the derivatives are bounded uniformly in $t$ by some constant $L>0$.
\end{assum}
Note that we can use the results established in \cite{Avikainen} and (\ref{eq2}) to derive the following bound
$$|Eg(X_{\alpha,t}) -Eg(X_{\beta,t})|\leq C|\alpha-\beta|^{1-\varepsilon}\,\,\forall\,\alpha,\beta\in (\frac{1}{2}, 1],$$
where $\varepsilon $ is some positive constant. However, this estimate is sub-optimal when $\alpha\to\beta.$ In the next theorem, we provide an optimal bound for $|Eg(X_{\alpha,t}) -Eg(X_{\beta,t})|.$ The price to pay is that we require $\alpha \vee \beta\in [\frac{7}{8},1].$
\begin{theorem}\label{dlc}Suppose Assumptions \ref{asum1} and \ref{asum3}. Let $(X_{\alpha,t})_{t\in [0,T]}$ and $(X_{\beta,t})_{t\in [0,T]}$ be solution to the equations (\ref{eq1}) and  (\ref{eq1qc}), respectively. We assume additionally that
	$$ |\sigma(t,x)|\ge \sigma_0 > 0\mbox{  }\forall\, t\in[0,T], x\in \mathbb{R}. $$ Let $g$ be a measurable function with $\|g\|_{\infty} <\infty$  and $\alpha_0\in (\frac{1}{2}, 1].$ Then, for any $\alpha\in [\alpha_0,1]$ and $\beta\in [\frac{7}{8},1],$ we have
		\begin{align}\label{pt2.1}
		|Eg(X_{\alpha,t}) -Eg(X_{\beta,t})| \le \|g\|_{\infty}t^{\alpha\wedge\beta-\beta}(|\ln t|+1)|\alpha -\beta|\,\,\forall\,t\in (0,T],
		\end{align}
			where $C$ is a positive constant depending only on $L, T, p,$ $\alpha_0$ and $x_0.$ Furthermore, when $g$ is continuous, it holds that
		\begin{equation}\label{pt2.2}\lim\limits_{\alpha\to \beta}\frac{Eg(X_{\alpha,t}) -Eg(X_{\beta,t})}{\alpha-\beta}=E\left[g(X_{\beta,t})\delta\left(\frac{Y_{\beta,t}DX_{\beta,t}}{\|DX_{\beta,t}\|_{L^2[0,T]}}\right)\right],
		\end{equation}
		where $(Y_{\beta,t})_{t\in [0,T]}$ is defined by \eqref{eq1m}.
	\end{theorem}
\subsection{Malliavin derivative of the solutions}
The proof of Theorem \ref{dlc} is based on Lemma \ref{dltq}. Hence, in this subsection,  we investigate Malliavin derivatives of the solution to the equation (\ref{eq1}).
	\begin{proposition}\label{pro2}Suppose Assumptions \ref{asum1} and \ref{asum3}. Let $(X_{\alpha,t})_{t\in [0,T]}$  be solution to the equation (\ref{eq1}). Then, for each $t\in[0,T],$ $X_{\alpha,t}$ is a twice Malliavin differentiable random variable. Moreover, the Malliavin derivative $D_\theta X_{\alpha,t}$ satisfies $D_\theta X_{\alpha,t}=0$ for $\theta>t$ and for $  0 \leq r,\theta < t,$
		\begin{align}
		D_{\theta}	X_{\alpha,t}= \frac{1}{\Gamma(\alpha)}\sigma(\theta,X_{\alpha,\theta} )(t-\theta)^{\alpha -1}& + \frac{1}{\Gamma(\alpha)}
		\int_{\theta}^t (t-s)^{\alpha -1}b'(s,X_{\alpha,s})D_{\theta}X_{\alpha,s}ds\notag \\
		&+\frac{1}{\Gamma(\alpha)}\int_{\theta}^t (t-s)^{\alpha -1} \sigma'(s,X_{\alpha,s})D_{\theta}X_{\alpha,s}dB_s,\label{dh}
		\end{align}
		and
\begin{multline}\label{dh2}
		D_rD_{\theta}X_{\alpha,t} = \frac{1}{\Gamma(\alpha)} \big( (t-r)^{\alpha -1}\sigma'(r,X_{\alpha,r})D_{\theta}X_{\alpha,r} +  (t- \theta)^{\alpha -1}\sigma'(\theta,X_{\alpha,\theta}) D_rX_{\alpha, \theta}\big)\\
		+ \frac{1}{\Gamma(\alpha)}\int_{r \vee \theta}^t  (t-s)^{\alpha -1}\left[b''(s,X_{\alpha,s})D_rX_{\alpha,s}D_{\theta}X_{\alpha,s}+ b'(s,X_{\alpha,s})D_rD_{\theta}X_{\alpha,s} \right]ds\\
		+ \frac{1}{\Gamma(\alpha)} \int_{r \vee \theta}^t (t-s)^{\alpha -1} \left[\sigma''(s,X_{\alpha,s})D_rX_{\alpha,s}D_{\theta}X_{\alpha,s} +  \sigma'(s,X_{\alpha,s})D_rD_{\theta}X_{\alpha,s}\right]dB_s.
 		\end{multline}
	\end{proposition}
	\begin{proof}
		We consider the Picard approximation sequence ${X_{\alpha,t}^n}$ defined by $X^0_{\alpha,t} = x_0$ and for $n\geq 0,$
	$$ X^{n+1}_{\alpha,t} = x_0 + \frac{1}{\Gamma(\alpha)} \bigg[\int_0^t (t -s)^{\alpha -1} b(s, X^{n}_{\alpha,s})ds +\int_0^t (t -s)^{\alpha -1} \sigma(s, X^{n}_{\alpha,s})dB_s \bigg],\,\,0\leq t\leq T.
$$
It is well known that $X^{n}_{\alpha,t}  \rightarrow X_{\alpha,t} $  in $L^2(\Omega)$ as $n\to \infty$ (this is an immediate consequence of the contraction mapping argument presented in the proof of Theorem 3.3 in \cite{Abi2019a}).  We next prove that
$$\sup\limits_{n \ge 0} E\|D_{\theta} X^{n}_{\alpha,t}\|^2_{L^2[0,T]}< \infty,\,\,0\leq t\leq T.$$
		It is easy to see that $X^n_{\alpha,t}$ is Malliavin differentiable for all $n\ge 0$ and $t\in[0,T].$ Moreover, we have
		\begin{align*}
		D_{\theta}X^{n+1}_{\alpha,t} = \frac{1}{\Gamma(\alpha)} \bigg[
		\sigma(\theta,X^{n}_{\alpha,\theta} )(t-\theta)^{\alpha -1}&+\int_{\theta}^t (t-s)^{\alpha -1}b'(s,X^n_{\alpha,s})D_{\theta}X^n_{\alpha,s}ds\\
		&+\int_{\theta}^t (t-s)^{\alpha -1} \sigma'(s,X^n_{\alpha,s})D_{\theta}X^n_{\alpha,s}dB_s \bigg],\,\,0\leq t\leq T.
		\end{align*}
		By  the Cauchy-Schwarz inequality, we obtain
\begin{align*}
E\|D_{\theta}&X^{n+1}_{\alpha,t}\|^2_{L^2[0,T]}=\int_0^t E|D_{\theta}X^{n+1}_{\alpha,t}|^2d\theta\\
&\le 3\int_0^t E\sigma^2(\theta,X^{n}_{\alpha,\theta} )(t-\theta)^{2\alpha -2}d\theta +  3t\int_0^t \int_{\theta}^t (t-s)^{2\alpha -2}E|b'(s,X^n_{\alpha,s})D_{\theta}X^n_{\alpha,s}|^2ds d\theta \\
		& +3\int_0^t\int_{\theta}^t (t-s)^{2\alpha -2} E|\sigma'(s,X^n_{\alpha,s})D_{\theta}X^n_{\alpha,s}|^2dsd\theta,\,\,0\leq\theta< t\leq T.
		\end{align*}
By Assumption \ref{asum1} and the estimate (\ref{pt2}) 
\begin{equation}\label{q}
E\|D_{\theta}X^{n+1}_{\alpha,t}\|^2_{L^2[0,T]}\le C +  C\int_0^t (t-s)^{2\alpha -2}\int_{0}^s E|D_{\theta}X^n_{\alpha,s}|^2d\theta ds,\,\,0\leq t\leq T,
\end{equation}
where $C$ is a positive constant not depending on $t$ and $n.$ Put $u_t:=  \sup\limits_{n\geq 0}(E\|D_{\theta}X^n_{\alpha,t}\|^2_{L^2[0,T]})^p$. Taking supremum the two sides of (\ref{q}) we get
		\begin{align*}
		u_t \le C + C\int_0^t (t-s)^{2\alpha -2}u_sds,\,\,0\leq t\leq T.
		\end{align*}
So, by the inequality (\ref{dhk}) with  $\eta=2\alpha-1,$ we get
$$u(t)\leq 2C\exp\left(\frac{4^{\frac{2}{\eta}}C^{\frac{2}{\eta}}t^2}{\eta^{\frac{2}{\eta}}}\right),\,\,0\leq t\leq T.$$
This implies that
$$\sup\limits_{n \ge 0} E\|D_{\theta} X^{n}_{\alpha,t}\|^2_{L^2[0,T]}< \infty,\,\,0\leq t\leq T.$$	
We now use Lemma 1.5.3 in \cite{nualartm2} to conclude that, for any $t\in [0,T],$ the random variable $X_{\alpha,t}$ is Malliavin differentiable and its Malliavin derivative of $X_{\alpha,t}$ is given by (\ref{dh}).

By using the same arguments as above, we can deduce for each $t\in[0,T]$, $X_{\alpha,t}$ is a twice Malliavin differentiable random variable. The proof of the proposition is complete.
	\end{proof}
	\begin{proposition}\label{pro3}Assume that Assumptions \ref{asum1} and \ref{asum3} hold. Let $(X_{\alpha,t})_{t\in [0,T]}$  be solution to the equation (\ref{eq1}). Let $\alpha_0\in (\frac{1}{2},1).$ Then, for every $p \in [2,\frac{1}{1-\alpha_0}]$ and $\alpha\in \left[\alpha_0,1\right],$ we have
		\begin{align}\label{dtheta}
		\int_0^t E|D_{\theta}X_{\alpha,t}|^pd\theta \le C t^{(\alpha -1)p+1}\,\,\forall\,0 \le  t\le T,
		\end{align}
where $C$ is a positive constant depending only on  $L, T, p,$ $x_0$ and $\alpha_0.$
	\end{proposition}
	\begin{proof}
For any $p\geq 2,$	by the inequality \eqref{pt1}, it follows from (\ref{dh}) that
		\begin{align*}
		E|D_{\theta}X_{\alpha,t}|^p  \le  3^{p-1} \bigg[E|\sigma(\theta,X_{\alpha,\theta} )|^p&(t-\theta)^{(\alpha -1)p} + E\bigg|\int_{\theta}^t (t-s)^{\alpha -1}b'(s,X_{\alpha,s})D_{\theta}X_{\alpha,s}ds\bigg|^p\\
		&+ E\bigg|\int_{\theta}^t (t-s)^{\alpha -1} \sigma'(s,X_{\alpha,s})D_{\theta}X_{\alpha,s}dB_s\bigg|^p \bigg],\,\,0 \le \theta < t\le T.
		\end{align*}
		By the estimate (\ref{pt2}) and linear growth property of $\sigma,$ one can derive that
$$E[|\sigma(\theta,X_{\alpha,\theta} )|^p(t-\theta)^{(\alpha -1)p}]\le C(t-\theta)^{(\alpha -1)p},\,\,0 \le \theta < t\le T,$$
where $C$ is a positive constant depending only on $L, T, p,$ $\alpha_0$ and $x_0.$ In addition, by Assumption \ref{asum3}, the  H\"{o}lder inequality and   Burkh\"older-Davis-Gundy inequality, we have
		\begin{align*}
		E\bigg|\int_{\theta}^t (t-s)^{\alpha -1}b'(s,X_{\alpha,s})&D_{\theta}X_{\alpha,s}ds\bigg|^p
  \le L(t-\theta)^{\frac{p}{2}} \bigg(\int_{\theta}^t (t-s)^{2\alpha -2} E|D_{\theta}X_{\alpha,s}|^2ds\bigg)^{\frac{p}{2}}\\
		& \le L t^{\frac{p}{2}}\bigg(\int_{\theta}^t(t-s)^{2\alpha -2} ds\bigg)^{\frac{p}{2}-1} \int_{\theta}^t(t-s)^{2\alpha -2} E|D_{\theta}X_{\alpha,s}|^pds\\
		& \le Lt^{\frac{p}{2}}\left(\frac{t^{2\alpha-1}}{2\alpha-1}\right)^{\frac{p}{2}-1} \int_{\theta}^t(t-s)^{2\alpha -2} E|D_{\theta}X_{\alpha,s}|^pds\\&
		\leq C\int_{\theta}^t(t-s)^{2\alpha -2} E|D_{\theta}X_{\alpha,s}|^pds,\,\,0 \le \theta < t\le T.
		\end{align*}
		and, for some $c_p>0,$
		\begin{align*}
		E\bigg|\int_{\theta}^t (t-s)^{\alpha -1} \sigma'(s,X_{\alpha,s})&D_{\theta}X_{\alpha,s}dB_s\bigg|^p \le Lc_p E\bigg(\int_{\theta}^t  (t-s)^{2\alpha -2} |D_{\theta}X_{\alpha,s}|^2ds\bigg)^{\frac{p}{2}}\\
		&  \le Lc_p \bigg(\int_\theta^t(t-s)^{2\alpha -2} ds\bigg)^{\frac{p}{2}-1} \int_{\theta}^t(t-s)^{2\alpha -2} E|D_{\theta}X_{\alpha,s}|^pds\\
		& \le Lc_p \left(\frac{t^{2\alpha-1}}{2\alpha-1}\right)^{\frac{p}{2}-1} \int_{\theta}^t(t-s)^{2\alpha -2} E|D_{\theta}X_{\alpha,s}|^pds\\&
		\leq C\int_{\theta}^t(t-s)^{2\alpha -2} E|D_{\theta}X_{\alpha,s}|^pds,\,\,0 \le \theta < t\le T,
		\end{align*}
		where  $C$ is a positive constant depending  only on  $L, T, p,$ $\alpha_0$ and $x_0.$	We deduce
		\begin{align*}
		E|D_{\theta}X_{\alpha,t}|^p  \le C(t-\theta)^{(\alpha -1)p} +  C \int_{\theta}^t(t-s)^{2\alpha -2} E|D_{\theta}X_{\alpha,s}|^pds,\,\,0 \le \theta < t\le T.
		\end{align*}
Fixed $p \in [2,\frac{1}{1-\alpha_0}],$ we obtain
\begin{align*}
\int_0^t E|D_{\theta}X_{\alpha,t}|^pd\theta  \le Ct^{(\alpha -1)p+1} +  C \int_0^t(t-s)^{2\alpha -2} \int_0^sE|D_{\theta}X_{\alpha,s}|^pd\theta ds,\,\,0 \le \theta < t\le T.
\end{align*}
Since $w(t)=t^{(\alpha -1)p+1},0\leq t\leq T$ is an increasing function, this allows us to use the inequality (\ref{dhk}) with  $\eta=2\alpha-1$ and we obtain the desired estimate (\ref{dtheta}).

The proof  of the proposition is complete.
	\end{proof}
	\begin{proposition}\label{pro5}Suppose that Assumptions \ref{asum1} and \ref{asum3} hold. Let $(X_{\alpha,t})_{t\in [0,T]}$  be solution to the equation (\ref{eq1}). 	Let $\alpha_0\in (\frac{3}{4},1).$ Then, for every $p \in (1,\frac{1}{4-4\alpha_0}]$ and $\alpha\in \left[\alpha_0,1\right],$ we have
		$$E\left[ \int_0^t\int_0^t|D_rD_{\theta}X_{\alpha,t}|^2drd\theta\right]^p\leq Ct^{2p(2\alpha-1)},\,\,0\leq t\leq T, $$
where  $C$ is a positive constant depending  only on  $L, T, p,$ $\alpha_0$ and $x_0.$	
	\end{proposition}
	\begin{proof} For any $p>1,$ using Assumption \ref{asum3}, the  H\"{o}lder  and   Burkh\"older-Davis-Gundy inequalities, it follows from \eqref{dh2} that	
		\begin{align*}
		E|D_rD_{\theta}&X_{\alpha,t}|^{2p}
\le C (t-r)^{2p(\alpha-1)}E|D_{\theta}X_{\alpha,r}|^{2p} + C(t- \theta)^{2p(\alpha-1)}E|D_rX_{\alpha, \theta}|^{2p}\\
		&+CE\left(\int_{r\vee \theta}^t (t-s)^{2\alpha-2}\left(|D_rX_{\alpha,s}D_{\theta}X_{\alpha,s}| +|D_rD_{\theta}X_{\alpha,s}|\right)^2ds\right)^{p}\\
&\le C (t-r)^{2p(\alpha-1)}E|D_{\theta}X_{\alpha,r}|^{2p} + C(t- \theta)^{2p(\alpha-1)}E|D_rX_{\alpha, \theta}|^{2p}\\
&+C\left(\int_{r\vee \theta}^t (t-s)^{2\alpha-2}ds\right)^{p-1}\left(\int_{r\vee \theta}^t (t-s)^{2\alpha-2}E\left(|D_rX_{\alpha,s}D_{\theta}X_{\alpha,s}| +|D_rD_{\theta}X_{\alpha,s}|\right)^{2p}ds\right)
\end{align*}
for all $0 \leq r,\theta < t\leq T,$ where  $C$ is a positive constant depending  only on  $L, T,$ and  $p.$ Note that $E|D_rX_{\alpha,s}D_{\theta}X_{\alpha,s}|^{2p}\leq \sqrt{E|D_rX_{\alpha,s}|^{4p}}\sqrt{E|D_{\theta}X_{\alpha,s}|^{4p}}.$ Then, we deduce
\begin{align}
& \int_0^t\int_0^tE|D_rD_{\theta}X_{\alpha,t}|^{2p}drd\theta\notag\\
&\le C \int_0^t(t-r)^{2p(\alpha-1)}\int_0^rE|D_{\theta}X_{\alpha,r}|^{2p}d\theta dr + C\int_0^t(t- \theta)^{2p(\alpha-1)}\int_0^\theta E|D_rX_{\alpha, \theta}|^{2p}drd\theta\notag\\
&+C\left(\int_{r\vee \theta}^t (t-s)^{2\alpha-2}ds\right)^{p-1}\int_0^t (t-s)^{2\alpha-2}\int_0^s\sqrt{E|D_rX_{\alpha,s}|^{4p}}dr\int_0^s\sqrt{E|D_{\theta}X_{\alpha,s}|^{4p}}d\theta ds\notag\\
&+C\int_{0}^t (t-s)^{2\alpha-2}\left(\int_0^s\int_0^sE|D_rD_{\theta}X_{\alpha,s}|^{2p}drd\theta\right)ds.\notag
\end{align}
Fixed $p \in (1,\frac{1}{4-4\alpha_0}].$ Recalling the estimate (\ref{dtheta}) gives us
\begin{align}
& \int_0^t\int_0^tE|D_rD_{\theta}X_{\alpha,t}|^{2p}drd\theta\notag\\
&\le C \int_0^t(t-r)^{2p(\alpha-1)}r^{2p(\alpha-1)+1} dr + C\int_0^t(t- \theta)^{2p(\alpha-1)}\theta^{2p(\alpha-1)+1} d\theta\notag\\
&+C\int_0^t (t-s)^{2\alpha-2}s^{4p(\alpha-1)+2} ds+C\int_{0}^t (t-s)^{2\alpha-2}\left(\int_0^s\int_0^sE|D_rD_{\theta}X_{\alpha,s}|^{2p}drd\theta\right)ds,\label{7yi}
\end{align}
where  $C$ is a positive constant depending  on $L, T, p,$ $\alpha_0$ and $x_0$. 	Next, we observe that
$$\int_0^t(t-r)^{2p(\alpha-1)}r^{2p(\alpha-1)+1} dr=t^{4p(\alpha-1)+2} \int_0^1(1-u)^{2p(\alpha-1)}u^{2p(\alpha-1)+1}du$$
and
$$\int_0^t (t-s)^{2\alpha-2}s^{4p(\alpha-1)+2} ds=t^{(4p+2)(\alpha-1)+3} \int_0^t(1-u)^{2\alpha-2}u^{4p(\alpha-1)+2}du$$
Furthermore, $\int_0^1(1-u)^{2p(\alpha-1)}u^{2p(\alpha-1)+1}du=\frac{\Gamma(2p(\alpha-1)+1)\Gamma(2p(\alpha-1)+2)}{\Gamma(4p(\alpha-1)+3)}$ and $\int_0^1(1-u)^{2\alpha-2}u^{4p(\alpha-1)+2}du=\frac{\Gamma(2\alpha-1)\Gamma(4p(\alpha-1)+3)}{\Gamma((4p+2)(\alpha-1)+4)}.$ Hence, for every $p \in (1,\frac{1}{4-4\alpha_0}],$ we can use the part $(ii)$ of Lemma \ref{l1} to verify that the integrals $\int_0^1(1-u)^{2p(\alpha-1)}u^{2p(\alpha-1)+1}du$ and $\int_0^1(1-u)^{2\alpha-2}u^{4p(\alpha-1)+2}du$ are bounded uniformly in $\alpha\in [\alpha_0,1].$  So we obtain from (\ref{7yi}) that
		\begin{align*}
		 \int_0^t\int_0^tE|D_rD_{\theta}X_{\alpha,t}|^{2p}drd\theta \le Ct^{4p(\alpha-1)+2}+C\int_{0}^t (t-s)^{2\alpha-2}\left(\int_0^s\int_0^sE|D_rD_{\theta}X_{\alpha,s}|^{2p}drd\theta\right)ds,
		\end{align*}where  $C$ is a positive constant depending  on $L, T, p,$ $\alpha_0$ and $x_0$. Therefore, by using the inequality (\ref{dhk}) with  $\eta=2\alpha-1,$  we obtain
$$	\int_0^t\int_0^tE|D_rD_{\theta}X_{\alpha,t}|^{2p}drd\theta\leq Ct^{4p(\alpha-1)+2},\,\,0\leq t\leq T,$$
	where $p \in (1,\frac{1}{4-4\alpha_0}]$ and $C$ is a positive constant depending  only on $L, T, p,$ $\alpha_0$ and $x_0$.	Using the H\"{o}lder inequality, we deduce
		\begin{align*}
		E\left[ \int_0^t\int_0^t|D_rD_{\theta}X_{\alpha,t}|^2drd\theta\right]^p\le Ct^{2(p-1)}\int_0^t\int_0^tE |D_rD_{\theta}X_{\alpha,t}|^{2p}drd\theta \le Ct^{2p(2\alpha-1)}.
		\end{align*}
		The proof of the proposition is complete.
	\end{proof}
\subsection{Proof of Theorem \ref{dlc}}
The proof of Theorem \ref{dlc} will be given at the end of this subsection. In order to be able to apply Lemma \ref{dltq}, we first need to prepare some technical results.
	\begin{proposition}\label{pro4}Assume that Assumptions \ref{asum1} and \ref{asum3} hold. Let $(X_{\alpha,t})_{t\in [0,T]}$ and $(X_{\beta,t})_{t\in [0,T]}$ be the solutions to the equations (\ref{eq1}) and  (\ref{eq1qc}), respectively. Then, for  every $p\geq 2$ and $\alpha_0\in (\frac{1}{2}, 1],$ we have
		$$ E\| DX_{\alpha,t}-DX_{\beta,t}\|_{L^2[0,T]}^2\leq C t^{2(\alpha\wedge\beta)-1}(|\ln t|^2+1)|\alpha-\beta|^2\,\,\forall\,t\in (0,T],\alpha,\beta\in [\alpha_0,1], $$
			where  $C$ is a positive constant depending  on $L,T,\alpha_0$ and $x_0.$
	\end{proposition}
	\begin{proof}
		We have
		\begin{align*}
		D_{\theta}X_{\alpha,t}-D_{\theta}X_{\beta,t}&=  \frac{1}{\Gamma(\alpha)} \bigg[
		\sigma(\theta,X_{\alpha,\theta} )(t-\theta)^{\alpha -1}+\int_{\theta}^t (t-s)^{\alpha -1}b'(s,X_{\alpha,s})D_{\theta}X_{\alpha,s}ds \bigg]\\
		& +\frac{1}{\Gamma(\alpha)}\int_{\theta}^t (t-s)^{\alpha -1} \sigma'(s,X_{\alpha,s})D_{\theta}X_{\alpha,s}dB_s \\
		& -  \frac{1}{\Gamma(\beta)} \bigg[
		\sigma(\theta,X_{\beta,\theta} )(t-\theta)^{\beta -1}+\int_{\theta}^t (t-s)^{\beta -1}b'(s,X_{\beta,s})D_{\theta}X_{\beta,s}ds\bigg]\\
		& -\frac{1}{\Gamma(\beta)}\int_{\theta}^t (t-s)^{\beta -1} \sigma'(s,X_{\beta,s})D_{\theta}X_{\beta,s}dB_s \bigg] \\
		& = 	J_1(\theta)+	J_2(\theta)+	J_3(\theta),\,\,0\leq t\leq T,
		\end{align*}
		where
		\begin{align*}
		&	J_1(\theta) = \frac{1}{\Gamma(\alpha)}\sigma(\theta,X_{\alpha,\theta} )(t-\theta)^{\alpha -1} -\frac{1}{\Gamma(\beta)} \sigma(\theta,X_{\beta,\theta} )(t-\theta)^{\beta -1},\\
		& J_2(\theta) =  \frac{1}{\Gamma(\alpha)}\int_{\theta}^t (t-s)^{\alpha -1}b'(s,X_{\alpha,s})D_{\theta}X_{\alpha,s}ds -\frac{1}{\Gamma(\beta)} \int_{\theta}^t (t-s)^{\beta -1}b'(s,X_{\beta,s})D_{\theta}X_{\beta,s}ds,\\
		&J_3(\theta) = \frac{1}{\Gamma(\alpha)} \int_{\theta}^t (t-s)^{\alpha -1} \sigma'(s,X_{\alpha,s})D_{\theta}X_{\alpha,s}dB_s - \frac{1}{\Gamma(\beta)}\int_{\theta}^t (t-s)^{\beta -1} \sigma'(s,X_{\beta,s})D_{\theta}X_{\beta,s}dB_s.
		\end{align*}
		By the  inequality \eqref{pt1}, we  have
		\begin{align*}
		E\|	DX_{\alpha,t}-DX_{\beta,t}\|_{L^2[0,T]}^2& = \int_0^t E|	D_{\theta}X_{\alpha,t}-D_{\theta}X_{\beta,t}|^2 d\theta\\
& \le 3\int_0^t(E|J_1(\theta)|^2+ E|J_2(\theta)|^2+E|J_3(\theta)|^2)d\theta,\,\,0\leq t\leq T.
		\end{align*}
		We next observe that
		\begin{align*}
		J_1(\theta)=   \frac{1}{\Gamma(\alpha)}\sigma(\theta,X_{\alpha,\theta} ) \Big((t-\theta)^{\alpha-1} -(t-\theta)^{\beta-1}\Big)+ \frac{1}{\Gamma(\alpha)}(t-\theta)^{\beta-1}\Big(\sigma(\theta,X_{\alpha,\theta} )-\sigma(\theta,X_{\beta,\theta} )\Big)\\
		+\bigg( \frac{1}{\Gamma(\alpha)} -  \frac{1}{\Gamma(\beta)} \bigg)\sigma(\theta,X_{\beta,\theta} )(t-\theta)^{\beta -1}.
		\end{align*}
Then, 	using the  inequality \eqref{pt1} and Assumption \ref{asum1}, we deduce
		\begin{align*}
		\int_0^t E|J_1(\theta)|^2d\theta &\le \frac{3}{\Gamma^2(\alpha)}\int_0^t \Big((t-\theta)^{\alpha-1} -(t-\theta)^{\beta-1}\Big)^2 E|\sigma(\theta,X_{\alpha,\theta} )|^2d\theta\\& \ \ +\frac{3}{\Gamma^2(\alpha)}\int_0^t (t-\theta)^{2\beta-2}E\Big(\sigma(\theta,X_{\alpha,\theta} )-\sigma(\theta,X_{\beta,\theta} )\Big)^2d\theta\\& \ \ +3\bigg( \frac{1}{\Gamma(\alpha)} -  \frac{1}{\Gamma(\beta)} \bigg)^2\int_0^t(t-\theta)^{2\beta -2}E|\sigma(\theta,X_{\beta,\theta} )|^2d\theta\\
		& \le 6L^2\int_0^t \Big((t-\theta)^{\alpha-1} -(t-\theta)^{\beta-1}\Big)^2 (1+E|X_{\alpha,\theta} |^2)d\theta\\& \ \
+3L^2\int_0^t (t-\theta)^{2\beta-2}E|X_{\alpha,\theta} -X_{\beta,\theta}|^2d\theta\\& \ \ +3L^2\bigg( \frac{1}{\Gamma(\alpha)} -  \frac{1}{\Gamma(\beta)} \bigg)^2\int_0^t(t-\theta)^{2\beta -2}(1+E|X_{\alpha,\theta} |^2)d\theta.
		\end{align*}
		This, together with the estimates obtained in Lemmas \ref{l2} and \ref{lm3} and  Proposition \ref{pro1}, yields
		\begin{align*}
		\int_0^t E|J_1(\theta)|^2d\theta &\le  C t^{2(\alpha\wedge\beta)-1}(|\ln t|^2+1)|\alpha-\beta|^2+C|\alpha-\beta|^2\int_0^t(t-\theta)^{2\beta -2}d\theta\\&
		\leq C t^{2(\alpha\wedge\beta)-1}(|\ln t|^2+1)|\alpha-\beta|^2+C|\alpha-\beta|^2\frac{t^{2\beta-1}}{2\beta-1}\\&	\leq C t^{2(\alpha\wedge\beta)-1}(|\ln t|^2+1)|\alpha-\beta|^2,
		\end{align*}
		where  $C$ is a positive constant depending only on $L, T,\alpha_0$ and $x_0.$
		
For $J_2(\theta)$, we have the decomposition
\begin{align*}
		J_2(\theta) &=\frac{1}{\Gamma(\alpha)} \int_{\theta}^t \bigg((t-s)^{\alpha -1}b'(s,X_{\alpha,s})D_{\theta}X_{\alpha,s} -(t-s)^{\beta -1}b'(s,X_{\beta,s})D_{\theta}X_{\beta,s}\bigg)ds \\
		&+\bigg( \frac{1}{\Gamma(\alpha)} -  \frac{1}{\Gamma(\beta)} \bigg)\int_{\theta}^t (t-s)^{\beta -1}b'(s,X_{\beta,s})D_{\theta}X_{\beta,s}ds\\
		& = \frac{1}{\Gamma(\alpha)} \int_{\theta}^t  (t-s)^{\alpha -1}b'(s,X_{\alpha,s}) (D_{\theta}X_{\alpha,s} - D_{\theta}X_{\beta,s})ds\\
		& + \frac{1}{\Gamma(\alpha)} \int_{\theta}^t (t-s)^{\alpha -1} (b'(s,X_{\alpha,s}) -b'(s,X_{\beta,s}))D_{\theta}X_{\beta,s} ds\\
		& + \frac{1}{\Gamma(\alpha)} \int_{\theta}^t \left((t-s)^{\alpha -1} -(t-s)^{\beta -1}\right)b'(s,X_{\beta,s})D_{\theta}X_{\beta,s} ds\\
		&+\bigg( \frac{1}{\Gamma(\alpha)} -  \frac{1}{\Gamma(\beta)} \bigg)\int_{\theta}^t (t-s)^{\beta -1}b'(s,X_{\beta,s})D_{\theta}X_{\beta,s}ds. 	
		\end{align*}
		By  the H\"{o}lder inequality and Assumption \ref{asum1} we get
		\begin{align*}
		\int_0^t E|J_2(\theta)|^2 d\theta &\le 4 \int_0^t(t-\theta)\left(\int_{\theta}^t (t-s)^{2\alpha -2}E|b'(s,X_{\alpha,s}) (D_{\theta}X_{\alpha,s} - D_{\theta}X_{\beta,s})|^2 ds\right) d\theta\\
		& +4 \int_0^t(t-\theta)\left(\int_{\theta}^t (t-s)^{2\alpha -2}E|(b'(s,X_{\alpha,s}) -b'(s,X_{\beta,s}))D_{\theta}X_{\beta,s}|^2 ds\right) d\theta \\
		& + 4 \int_0^t(t-\theta)\left(\int_{\theta}^t \left((t-s)^{\alpha -1} -(t-s)^{\beta -1}\right)^2E|b'(s,X_{\beta,s})D_{\theta}X_{\beta,s}|^2 ds\right) d\theta\\
		& + 4\bigg( \frac{1}{\Gamma(\alpha)} -  \frac{1}{\Gamma(\beta)} \bigg)^2  \int_0^t(t-\theta)\left(\int_{\theta}^t (t-s)^{2\beta -2}E|b'(s,X_{\beta,s})D_{\theta}X_{\beta,s}|^2ds\right) d\theta\\
		& \le 4L^2T \int_0^t\left(\int_{\theta}^t (t-s)^{2\alpha -2}E|D_{\theta}X_{\alpha,s} - D_{\theta}X_{\beta,s}|^2 ds\right) d\theta\\
		& +4L^2T  \int_0^t\left(\int_{\theta}^t (t-s)^{2\alpha -2}E|(X_{\alpha,s} -X_{\beta,s})D_{\theta}X_{\beta,s}|^2 ds\right) d\theta \\
		& + 4L^2T \int_0^t\left(\int_{\theta}^t \left((t-s)^{\alpha -1} -(t-s)^{\beta -1}\right)^2E|D_{\theta}X_{\beta,s}|^2 ds\right) d\theta\\
		& + 4T\bigg( \frac{1}{\Gamma(\alpha)} -  \frac{1}{\Gamma(\beta)} \bigg)^2  \int_0^t\left(\int_{\theta}^t (t-s)^{2\beta -2}E|D_{\theta}X_{\beta,s}|^2ds\right) d\theta,
		\end{align*}
and hence,
		\begin{align*}
		\int_0^t E|J_2(\theta)|^2 d\theta & \le
		4L^2T\int_0^t (t-s)^{2\alpha -2}\bigg(\int_0^s E|D_{\theta}X_{\alpha,s} - D_{\theta}X_{\beta,s}|^2 d\theta\bigg) ds\\
		& + 4L^2T \int_0^t (t-s)^{2\alpha -2} \bigg( \int_0^s \sqrt{ E |X_{\alpha,s} -X_{\beta,s}|^{4}} \sqrt{E |D_{\theta}X_{\beta,s}|^4}d\theta\bigg)ds\\
		& +   4L^2T\int_0^t\left((t-s)^{\alpha -1} -(t-s)^{\beta -1}\right)^2\left(\int_0^s E|D_{\theta}X_{\beta,s}|^2d\theta \right)ds\\
		&+ 4T\bigg( \frac{1}{\Gamma(\alpha)} -  \frac{1}{\Gamma(\beta)} \bigg)^2  \int_0^t(t-s)^{2\beta -2}\left(\int_{0}^s E|D_{\theta}X_{\beta,s}|^2d\theta\right) ds.
\end{align*}
We now use the estimates (\ref{eq2}) and (\ref{dtheta}) to get
\begin{align*}
		\int_0^t E|J_2(\theta)|^2 d\theta & \le
		C\int_0^t (t-s)^{2\alpha -2}\bigg(\int_0^s E|D_{\theta}X_{\alpha,s} - D_{\theta}X_{\beta,s}|^2 d\theta\bigg) ds\\
		& + C|\alpha-\beta|^2 \int_0^t (t-s)^{2\alpha -2}ds+   C\int_0^t\left((t-s)^{\alpha -1} -(t-s)^{\beta -1}\right)^2ds\\
		&+ C\bigg( \frac{1}{\Gamma(\alpha)} -  \frac{1}{\Gamma(\beta)} \bigg)^2  \int_0^t(t-s)^{2\beta -2}ds.
\end{align*}
where  $C$ is a positive constant depending  on $L, T,x_0$ and $\alpha_0.$ As a consequence, recalling the estimates (\ref{tfg}) and (\ref{hh2}), we obtain
		\begin{align*}
		\int_0^t E|J_2(\theta)|^2 d\theta &\le  C\int_0^t (t-s)^{2\alpha -2}\int_0^s E|D_{\theta}X_{\alpha,s} - D_{\theta}X_{\beta,s}|^2 d\theta ds+ C t^{2(\alpha\wedge\beta)-1}(|\ln t|^2+1)|\alpha-\beta|^2.
		\end{align*}
		Similarly, we also have
		\begin{align*}
		\int_0^t E|I_3(\theta)|^2 d\theta \le   C\int_0^t (t-s)^{2\alpha -2}\int_0^s E|D_{\theta}X_{\alpha,s} - D_{\theta}X_{\beta,s}|^2 d\theta ds+ C t^{2(\alpha\wedge\beta)-1}(|\ln t|^2+1)|\alpha-\beta|^2.
		\end{align*}
Combining the above computations yields
$$	E\|	DX_{\alpha,t}-DX_{\beta,t}\|_{L^2[0,T]}^2\le  C t^{2(\alpha\wedge\beta)-1}(|\ln t|^2+1)|\alpha-\beta|^2+ C\int_0^t (t-s)^{2\alpha -2} E\|	DX_{\alpha,s}-DX_{\beta,s}\|_{L^2[0,T]}^2 ds, $$
		where  $C$ is a positive constant depending  on $L,T,x_0$ and $\alpha_0.$ The latest estimate is similar to (\ref{caik}), and hence, we can get
\begin{align*}
E\|D_{\theta}X_{\alpha,t}-D_{\theta}X_{\beta,t}\|_{L^2[0,T]}^2& \le   C t^{2(\alpha\wedge\beta)-1}(|\ln t|^2+1)|\alpha -\beta|^2
\end{align*}
		for all $t\in [0,T]$.	The proof of the proposition is complete.
	\end{proof}

\begin{proposition}\label{pro6}
	Let Assumptions \ref{asum1} and \ref{asum3} hold. In addition, we assume that
	$$ |\sigma(t,x)|\ge \sigma_0 > 0\mbox{  }\forall\,t\in[0,T], x\in \mathbb{R}. $$
Let $\alpha_0\in (\frac{1}{2},1).$	Then, for every $\alpha\in \left[\alpha_0,1\right]$ and $\gamma\in(1,\frac{1}{2-2\alpha_0}],$ we have
	$$E\|DX_{\alpha,t}\|^{-2\gamma}_{L^2[0,T]}\le C t^{(1-2\alpha)\gamma},\,\,0<t\leq T,$$
where  $C$ is a positive constant depending only  on $L,T,x_0,\sigma_0$ and $\alpha_0.$
\end{proposition}
\begin{proof}
	By using the fundamental inequality $(a+b+c)^2 \ge \frac{a^2}{2} - 2(b^2+c^2),$ we obtain from (\ref{dh}) that
	\begin{align*}
	|D_{\theta}X_{\alpha,t}|^2 & \ge \frac{  (t-\theta)^{2\alpha -2}\sigma^2({\theta},X_{\alpha,{\theta}})}{2\Gamma^2(\alpha)} \\
	&- 2\bigg(\int_{\theta}^t \frac{(t-s)^{\alpha -1}}{\Gamma(\alpha)}b'(s,X_{\alpha,s})D_{\theta}X_{\alpha,s}ds \bigg)^2\\
	& - 2\bigg( \int_{\theta}^t \frac{(t-s)^{\alpha -1}}{\Gamma(\alpha)} \sigma'(s,X_{\alpha,s})D_{\theta}X_{\alpha,s}dB_s\bigg)^2.
	\end{align*}
Fixed $t\in (0,T].$	For each $y \ge y_0:= \frac{4(2\alpha-1)\Gamma^2(\alpha)}{\sigma_0^2t^{2\alpha-1}}$, the real number $\varepsilon=\frac{1}{t}\left( \frac{4(2\alpha-1)\Gamma^2(\alpha)}{y\sigma_0^2}\right)^{\frac{1}{2\alpha-1}}$ belongs to $(0,1]$. Hence,
	\begin{align*}
	\|DX_{\alpha,t}\|^2_{L^2[0,T]} &= \int_0^t|D_{\theta}X_{\alpha,t}|^2d{\theta}\ge \int_{t(1-\varepsilon)}^t |D_{\theta}X_{\alpha,t}|^2 d{\theta}\\
	& \ge \frac{\sigma^2_0}{2\Gamma^2(\alpha)} \int_{t(1-\varepsilon)}^t (t-{\theta})^{2\alpha -2}d\theta	- 2\int_{t(1-\varepsilon)}^t\bigg(\int_\theta^t\frac{(t-s)^{\alpha -1}}{\Gamma(\alpha)} b'(s,X_{\alpha,s})D_{\theta}X_{\alpha,s}ds \bigg)^2d\theta\\
	&\qquad\qquad\qquad -2 \int_{t(1-\varepsilon)}^t\bigg( \int_\theta^t\frac{(t-s)^{\alpha -1}}{\Gamma(\alpha)} \sigma'(s,X_{\alpha,s})D_{\theta}X_{\alpha,s}dB_s\bigg)^2d\theta\\
	& \ge \frac{\sigma^2_0 (t\varepsilon)^{2\alpha-1}}{2(2\alpha-1)\Gamma^2(\alpha)} - I_y(t)= \frac{2}{y} -I_y(t), 	
	\end{align*}
	where $I_y(t)$ is given by
	\begin{align*}
	I_y(t):= 2\int_{t(1-\varepsilon)}^t\bigg(\int_{\theta}^t\frac{(t-s)^{\alpha -1}}{\Gamma(\alpha)} b'(s,X_{\alpha,s})D_{\theta}X_{\alpha,s}ds \bigg)^2d\theta\\
	+2\int_{t(1-\varepsilon)}^t\bigg( \int_{\theta}^t\frac{(t-s)^{\alpha -1}}{\Gamma(\alpha)} \sigma'(s,X_{\alpha,s})D_{\theta}X_{\alpha,s}dB_s\bigg)^2d\theta.	
	\end{align*}
	Put $$A :=\int_{t(1-\varepsilon)}^t\bigg(\int_{\theta}^t\frac{(t-s)^{\alpha -1}}{\Gamma(\alpha)} b'(s,X_{\alpha,s})D_{\theta}X_{\alpha,s}ds \bigg)^2d\theta $$
	and $$B:= \int_{t(1-\varepsilon)}^t\bigg( \int_{\theta}^t\frac{(t-s)^{\alpha -1}}{\Gamma(\alpha)} \sigma'(s,X_{\alpha,s})D_{\theta}X_{\alpha,s}dB_s\bigg)^2d\theta.	
	$$
Fixed $q\in [2,\frac{1}{1-\alpha_0}].$ By the  H\"{o}lder inequality, we have
	\begin{align*}
	E|A|^{q/2} &\leq L^qE\bigg( \int_{t(1-\varepsilon)}^t \bigg(\int_{\theta}^t (t-s)^{\alpha -1}|D_{\theta}X_{\alpha,s}|ds\bigg)^2d\theta\bigg)^{\frac{q}{2}}\\
&\leq L^q(t\varepsilon)^{\frac{q}{2}-1} \int_{t(1-\varepsilon)}^t E\bigg(\int_{\theta}^t (t-s)^{\alpha -1}|D_{\theta}X_{\alpha,s}|ds\bigg)^qd\theta\\
&\leq L^q(t\varepsilon)^{\frac{q}{2}-1} \int_{t(1-\varepsilon)}^t \bigg(\int_{\theta}^t (t-s)^{\alpha -1}ds\bigg)^{q-1}\int_{\theta}^t (t-s)^{\alpha -1}E|D_{\theta}X_{\alpha,s}|^qdsd\theta\\
&\leq C(t\varepsilon)^{\frac{q}{2}-1}(t\varepsilon)^{\alpha(q-1)} \int_{t(1-\varepsilon)}^t (t-s)^{\alpha -1}\int_{t(1-\varepsilon)}^s E|D_{\theta}X_{\alpha,s}|^qd\theta ds.
	\end{align*}
By using the same arguments as in the proof of (\ref{dtheta}), we have
$$\int_{t(1-\varepsilon)}^s E|D_{\theta}X_{\alpha,s}|^qd\theta\leq C(s-t(1-\varepsilon))^{(\alpha-1)q+1}\leq (t\varepsilon)^{(\alpha-1)q+1},\,\,t(1-\varepsilon)\leq s\leq t.$$
So we obtain
$$E|A|^{q/2}\leq C(t\varepsilon)^{\frac{q}{2}-1}(t\varepsilon)^{\alpha(q-1)} \int_{t(1-\varepsilon)}^t (t-s)^{\alpha -1}(t\varepsilon)^{(\alpha-1)q+1} ds\leq C(t\varepsilon)^{\frac{q(4\alpha-1)}{2}}. $$
Similarly, we also have
	\begin{align*}
	E|B|^{q/2} & \le (t\varepsilon)^{\frac{q}{2}-1} \int_{t(1-\varepsilon)}^t E \bigg(\int_{\theta}^t (t-s)^{\alpha -1}\sigma'(s,X_{\alpha,s})D_{\theta}X_{\alpha,s}dB_s  \bigg)^qd\theta\\
	& \le C(t\varepsilon)^{\frac{q}{2}-1} \int_{t(1-\varepsilon)}^t E\bigg(\int_{\theta}^t (t-s)^{2\alpha -2}|D_{\theta}X_{\alpha,s}|^2ds\bigg)^{\frac{q}{2}}d\theta\\
& \le C(t\varepsilon)^{\frac{q}{2}-1} \int_{t(1-\varepsilon)}^t \bigg(\int_{\theta}^t (t-s)^{2\alpha -2}ds\bigg)^{\frac{q}{2}-1}\int_{\theta}^t (t-s)^{2\alpha -2}E|D_{\theta}X_{\alpha,s}|^qdsd\theta\\	
& \le C(t\varepsilon)^{\frac{q}{2}-1} (t\varepsilon)^{(2\alpha-1)(\frac{q}{2}-1)}\int_{t(1-\varepsilon)}^t (t-s)^{2\alpha -2}\int_{t(1-\varepsilon)}^s E|D_{\theta}X_{\alpha,s}|^qd\theta ds\\
& \le C(t\varepsilon)^{q(2\alpha-1)} .
	\end{align*}
	Consequently,
	\begin{align*}
	E|I_y(t)|^{\frac{q}{2}}  \le  2^{\frac{q}{2}-1}(E|A|^{\frac{q}{2}}+ E|B|^{\frac{q}{2}}) \le  C (t\varepsilon)^{q(2\alpha -1)},
	\end{align*}
where  $C$ is a positive constant depending only  on $L,T,x_0$ and $\alpha_0.$  By the Markov inequality, we get
	\begin{align*}
	P\bigg(\|DX_{\alpha,t}\|^2_{L^2[0,T]} \le \frac{1}{y}\bigg) &\le P\bigg(\frac{2}{y} -I_y(t) \le \frac{1}{y}\bigg) \notag\\
	&= P\bigg(I_y(t) \ge \frac{1}{y}\bigg) \le y^{q/2}E|I_y(t)|^{q/2}\\
& \le Cy^{q/2} (t\varepsilon)^{q(2\alpha -1)} \\
	& =   Cy^{q/2} t^{q(2\alpha -1)}\frac{1}{t^{q(2\alpha-1)}} \left(\frac{2(2\alpha-1)\Gamma^2(\alpha)}{y\sigma_0^2}\right)^q\\
	& \le Cy^{-\frac{q}{2}}\,\,\forall\,y\geq y_0,
	\end{align*}
where  $C$ is a positive constant depending only  on $L,T,x_0,\sigma_0$ and $\alpha_0.$
	
For any $\gamma\in(1,\frac{1}{2-2\alpha_0}],$ we put $2q:=\gamma+\frac{1}{2-2\alpha_0}.$ Then, $q\in [2,\frac{1}{1-\alpha_0}],$ and we obtain
	\begin{align*}
	E \|DX_{\alpha,t}\|^{-2\gamma}_{L^2[0,T]} &= \int_0^\infty \gamma y^{\gamma -1}P(\|DX_{\alpha,t}\|^{-2}_{L^2[0,T]} \ge y)dy\\
	&\le  \int_0^{y_0}\gamma y^{\gamma -1}dy+\int_{y_0}^\infty  \gamma y^{\gamma -1} P\big(\|DX_{\alpha,t}\|^{2}_{L^2[0,T]} \le \frac{1}{y}\big)dy\\
	& \le y_0^{\gamma} + \int_{y_0}^\infty  \gamma y^{\gamma -1} P\big(\|DX_{\alpha,t}\|^{2}_{L^2[0,T]} \le \frac{1}{y}\big)dy \\
	& \le y_0^{\gamma} + C \gamma \int_{y_0}^\infty y^{\gamma -1} y^{-q/2}dy\\
	& \le y_0^{\gamma} + C \gamma y_0^{\gamma-\frac{q}{2}}.
	\end{align*}
	Since $y_0= \frac{4(2\alpha-1)\Gamma^2(\alpha)}{\sigma_0^2t^{2\alpha-1}},$ we conclude that
	$$E \|DX_{\alpha,t}\|^{-2\gamma}_{L^2[0,T]}\le  C t^{(1-2\alpha)\gamma}, $$
where  $C$ is a positive constant depending only  on $L,T,x_0,\sigma_0$ and $\alpha_0.$	 The proof of the proposition is complete.
\end{proof}
	
\noindent {\it Proof of Theorem \ref{dlc}.} We will carry out the proof in two parts.

\noindent {\bf Part 1.} In this part, we prove \eqref{pt2.1}.  Fixed $t\in (0,T],$ an application of Lemma \ref{dltq} to $F_1=X_{\beta,t}$ and $F_2=X_{\alpha,t}$ gives us
\begin{align}
|Eg(X_{\alpha,t})&-Eg(X_{\beta,t})|\leq C\|g\|_\infty \|X_{\alpha,t}-X_{\beta,t}\|_{1,2}\notag\\
&\times\left(E\|DX_{\beta,t}\|^{-8}_{L^2[0,T]}E\left(\int_0^T\int_0^T|D_\theta D_rX_{\beta,t}|^2d\theta dr\right)^2+(E\|DX_{\beta,t}\|^{-2}_{L^2[0,T]})^2\right)^{\frac{1}{4}}.\notag
\end{align}
For any $\beta\in [\frac{7}{8},1],$ by using Proposition  \ref{pro5} with $\alpha_0=\frac{7}{8}$ and $p=2,$ we obtain
$$E\left(\int_0^T\int_0^T|D_\theta D_rX_{\beta,t}|^2d\theta dr\right)^2\leq Ct^{4(2\beta-1)}.$$
On the other hand, by using Proposition  \ref{pro6} with $\alpha_0=\frac{7}{8}$ and $\gamma=4,$ we have
$$E\|DX_{\beta,t}\|^{-8}_{L^2[0,T]}\leq C t^{4(1-2\beta)},$$
$$(E\|DX_{\beta,t}\|^{-2}_{L^2[0,T]})^2\leq C t^{2(1-2\beta)}.$$
Thanks to Proposition  \ref{pro4} we have
\begin{align*}
		\|X_{\alpha,t} - X_{\beta,t}\|_{1,2}&=\left(E|X_{\alpha,t} -X_{\beta,t}|^2+E\|D_{\theta}X_{\alpha,t}-D_{\theta}X_{\beta,t}\|^2_{L^2[0,T]}\right)^{1/2}\\
		&\le \left(Ct^{2(\alpha\wedge\beta)-1}(|\ln t|^2+1)|\alpha-\beta|^2\right)^{1/2}\le Ct^{\alpha\wedge\beta-\frac{1}{2}}(|\ln t|+1)|\alpha-\beta|.
		\end{align*}
Combining the above computations yields
		\begin{align*}
		\big| Eg(X_{\alpha,t}) - Eg(X_{\beta,t})\big|& \le C\|g\|_{\infty}\left(t^{4(1-2\beta)}t^{4(2\beta-1)}+t^{2(1-2\beta)}\right)^{\frac{1}{4}}t^{\alpha\wedge\beta-\frac{1}{2}}(|\ln t|+1)|\alpha -\beta| \\&\le C\|g\|_{\infty}t^{\alpha\wedge\beta-\beta}(|\ln t|+1)|\alpha -\beta|,
		\end{align*}
		which is the required estimate \eqref{pt2.1}.
		
\noindent{\bf	Part 2}. In this part, we prove \eqref{pt2.2}. We note that, by the estimates in Theorem \ref{dl32m} and Proposition \ref{pro4} , we have $\frac{X_{\alpha,t} -X_{\beta,t}}{\alpha-\beta}\to Y_{\beta,t}$ in $L^2(\Omega)$ as $\alpha\to \beta$
	 and $\max\limits_{\alpha\neq \beta}\frac{E\| DX_{\alpha,t}-DX_{\beta,t}\|^2_{L^2[0,T]}}{|\alpha-\beta|^2}<\infty.$
	 Thus, it follows from Lemma $1.2.3$ in \cite{nualartm2} that $Y_{\beta,t} \in \mathbb{D}^{1,2}$ for every $t\in [0,T], \beta\in[\alpha_0,1]$ and
	 $\frac{  DX_{\alpha,t}-DX_{\beta,t}}{\alpha-\beta}$ weakly converges to $DY_{\beta,t}$ in $L^2(\Omega \times [0, T ])$ as $\alpha \to \beta.$ By the  relationship \eqref{0jkd3}, we have
$$\delta\left(\frac{Y_{\beta,t}DX_{\beta,t}}{\|DX_{\beta,t}\|_{L^2[0,T]}}\right)
=Y_{\beta,t}\delta\left(\frac{DX_{\beta,t}}{\|DX_{\beta,t}\|_{L^2[0,T]}}\right)-\frac{\langle DY_{\beta,t},DX_{\beta,t}\rangle_{L^2[0,T]}}{\|DX_{\beta,t}\|_{L^2[0,T]}}.$$
Let $g$ be a bounded and continuous function. By the relation \eqref{oldl1} we have
	\begin{multline}\label{pt4m }	
Eg(X_{\alpha,t})-Eg(X_{\beta,t})\\
= E\left[\int_{X_{\beta,t}}^{X_{\alpha,t}} g(z)dz\delta\left(\frac{DX_{\beta,t} }{\|DX_{\beta,t}\|^2_{L^2[0,T]}}\right)\right]-E\left[\frac{g(X_{\alpha,t}) \langle DX_{\alpha,t}-DX_{\beta,t} , DX_{\beta,t} \rangle_{L^2[0,T]}}{\|DX_{\beta,t}\|_{L^2[0,T]}^2}\right].
	\end{multline}
Then, for $\alpha\not= \beta,$ we obtain
		\begin{align}\label{pt3.1.1} &\frac{Eg(X_{\alpha,t})-Eg(X_{\beta,t})}{\alpha-\beta}-E\left[g(X_{\beta,t})\delta\left(\frac{Y_{\beta,t}DX_{\beta,t}}{\|DX_{\beta,t}\|_{L^2[0,T]}}\right)\right]\notag\\&
	=\frac{1}{\alpha-\beta}E\left[\int_{X_{\beta,t}}^{X_{\alpha,t}} g(z)dz\delta\left(\frac{DX_{\beta,t} }{\|DX_{\beta,t}\|^2_{L^2[0,T]}}\right)\right] -\frac{1}{\alpha-\beta}E\left[\frac{g(X_{\alpha,t}) \langle DX_{\alpha,t}-DX_{\beta,t} , DX_{\beta,t} \rangle_{L^2[0,T]}}{\|DX_{\beta,t}\|_{L^2[0,T]}^2}\right]\notag\\&-E\left[g(X_{\beta,t})Y_{\beta,t}\delta\left(\frac{DX_{\beta,t}}{\|DX_{\beta,t}\|_{L^2[0,T]}}\right)\right]+E\left[\frac{g(X_{\beta,t}) \langle DY_{\beta,t} , DX_{\beta,t} \rangle_{L^2[0,T]}}{\|DX_{\beta,t} \|_{L^2[0,T]}^2}\right]\notag\\&
	=E\left[\left(\frac{1}{\alpha-\beta}\int_{X_{\beta,t}}^{X_{\alpha,t}} g(z)dz-g(X_{\beta,t})Y_{\beta,t}\right)\delta\left(\frac{DX_{\alpha,t} }{\|DX_{\beta,t}\|^2_{L^2[0,T]}}\right)\right]\notag\\&
	+E\left[\frac{(g(X_{\alpha,t})-g(X_{\beta,t}) ) \langle DX_{\alpha,t}-DX_{\beta,t} , DX_{\alpha,t} \rangle_{L^2[0,T]}}{(\alpha-\beta)\|DX_{\beta,t}\|_{L^2[0,T]}^2}\right]\notag\\&
	-E\left[g(X_{\beta,t})\left\langle \frac{  DX_{\alpha,t}-DX_{\beta,t}}{\alpha-\beta}-DY_{\beta,t},\frac{DX_{\beta,t}  }{\|DX_{\beta,t}\|_{L^2[0,T]}^2}\right\rangle_{L^2[0,T]}\right].
	\end{align}
We now observe that
		\begin{align*}
\frac{1}{\alpha-\beta}\int_{X_{\beta,t}}^{X_{\alpha,t}} g(z)dz-g(X_{\beta,t})Y_{\beta,t}&	=\frac{X_{\alpha,t}-X_{\beta,t}}{\alpha-\beta}\int_{0}^{1} g(X_{\beta,t}-z(X_{\alpha,t}-X_{\beta,t}))dz-g(X_{\beta,t})Y_{\beta,t}\\
&=\left(\frac{X_{\alpha,t}-X_{\beta,t}}{\alpha-\beta}-Y_{\beta,t}\right)\int_{0}^{1} g(X_{\beta,t}-z(X_{\alpha,t}-X_{\beta,t}))dz\\
&+Y_{\beta,t}\int_{0}^{1} (g(X_{\beta,t}-z(X_{\alpha,t}-X_{\beta,t}))-g(X_{\beta,t}))dz.
	\end{align*}
This implies
		\begin{align*}
E&\left|\left(\frac{1}{\alpha-\beta}\int_{X_{\beta,t}}^{X_{\alpha,t}} g(z)dz-g(X_{\beta,t})Y_{\beta,t}\right)\delta\left(\frac{DX_{\alpha,t} }{\|DX_{\beta,t}\|^2_{L^2[0,T]}}\right)\right|\\
&
	\leq\|g\|_\infty E\left|\left(\frac{X_{\alpha,t}-X_{\beta,t}}{\alpha-\beta}\right)\delta\left(\frac{DX_{\alpha,t} }{\|DX_{\beta,t}\|^2_{L^2[0,T]}}\right)\right|\\&+E\left|\delta\left(\frac{DX_{\alpha,t} }{\|DX_{\beta,t}\|^2_{L^2[0,T]}}\right)Y_{\beta,t}\int_{0}^{1} (g(X_{\beta,t}-z(X_{\alpha,t}-X_{\beta,t}))-g(X_{\beta,t}))dz\right|.
	\end{align*}
On the other hand,  by   the H\"{o}lder inequality and  Theorem \ref{dl32m} we get
	\begin{align*}E&\left|\left(\frac{X_{\alpha,t}-X_{\beta,t}}{\alpha-\beta}-Y_{\beta,t}\right)\delta\left(\frac{DX_{\alpha,t} }{\|DX_{\beta,t}\|^2_{L^2[0,T]}}\right)\right|\\&\leq \left(E\left|\frac{X_{\alpha,t}-X_{\beta,t}}{\alpha-\beta}-Y_{\beta,t}\right|^2\right)^{1/2}\left(E\left|\delta\left(\frac{DX_{\alpha,t} }{\|DX_{\beta,t}\|^2_{L^2[0,T]}}\right)\right|^2\right)^{1/2}\to 0 \ \mbox{as} \ \alpha\to \beta.\end{align*}
	Furthermore, noting that $g(X_{\beta,t}-z(X_{\alpha,t}-X_{\beta,t}))-g(X_{\beta,t})\to 0 \ \mbox{a.s. \ as} \ \alpha\to \beta.$ Using  the dominated convergence theorem we can get
	$$E\left|\delta\left(\frac{DX_{\alpha,t} }{\|DX_{\beta,t}\|^2_{L^2[0,T]}}\right)Y_{\beta,t}\int_{0}^{1} (g(X_{\beta,t}-z(X_{\alpha,t}-X_{\beta,t}))-g(X_{\beta,t}))dz\right|\to 0 \ \mbox{as} \ \alpha\to \beta.$$
	 So it holds that
	 \begin{equation}\label{pt3.1}E\left|\left(\frac{1}{\alpha-\beta}\int_{X_{\beta,t}}^{X_{\alpha,t}} g(z)dz-g(X_{\beta,t})Y_{\beta,t}\right)\delta\left(\frac{DX_{\alpha,t} }{\|DX_{\beta,t}\|^2_{L^2[0,T]}}\right)\right|\to 0 \ \mbox{as} \ \alpha\to \beta.
	 \end{equation}
	 Next, using  the H\"{o}lder inequality and Propositions \ref{pro4} and \ref{pro6}, we have
	 \begin{align*}E&\left[\frac{(g(X_{\alpha,t})-g(X_{\beta,t}) ) \langle DX_{\alpha,t}-DX_{\beta,t} , DX_{\alpha,t} \rangle_{L^2[0,T]}}{(\alpha-\beta)\|DX_{\beta,t}\|_{L^2[0,T]}^2}\right]\\
&\leq E\left[\frac{|g(X_{\alpha,t})-g(X_{\beta,t}) | \| DX_{\alpha,t}-DX_{\beta,t}\|_{L^2[0,T]}}{|\alpha-\beta|\|DX_{\beta,t}\|_{L^2[0,T]}}\right]\\&
	 \leq (E|g(X_{\alpha,t})-g(X_{\beta,t})|^4)^{1/4}\left(\frac{E\| DX_{\alpha,t}-DX_{\beta,t}\|^2_{L^2[0,T]}}{|\alpha-\beta|^2}\right)^{1/2}(E\|DX_{\beta,t}\|_{L^2[0,T]}^{-4})^{1/4}\\
&\leq C(E|g(X_{\alpha,t})-g(X_{\beta,t})|^4)^{1/4}.
	 \end{align*}
	 Once again, by the dominated convergence theorem, we have
	 \begin{multline}\label{pt3.2}E\left[\frac{(g(X_{\alpha,t})-g(X_{\beta,t}) ) \langle DX_{\alpha,t}-DX_{\beta,t} , DX_{\alpha,t} \rangle_{L^2[0,T]}}{(\alpha-\beta)\|DX_{\beta,t}\|_{L^2[0,T]}^2}\right]\\
\leq C(E|g(X_{\alpha,t})-g(X_{\beta,t})|^4)^{1/4}\to 0 \ \mbox{as} \ \alpha\to \beta.
	 \end{multline}
 It is easy to verify that the function $r\mapsto \frac{g(X_{\beta,t} )D_rX_{\beta,t}  }{\|DX_{\beta,t}\|_{L^2[0,T]}^2}$ is continuous. Hence, since  $\frac{  DX_{\alpha,t}-DX_{\beta,t}}{\alpha-\beta}$ weakly converges to $DY_{\beta,t}$ in $L^2(\Omega \times [0, T ]),$ we deduce
	  \begin{equation}\label{pt3.3}
	  E\left[g(X_{\beta,t})\left\langle \frac{  DX_{\alpha,t}-DX_{\beta,t}}{\alpha-\beta}-DY_{\beta,t},\frac{DX_{\beta,t}  }{\|DX_{\beta,t}\|_{L^2[0,T]}^2}\right\rangle_{L^2[0,T]}\right]\to 0 \ \mbox{as} \ \alpha\to \beta.
	  \end{equation}
	  Combining \eqref{pt3.1.1}-\eqref{pt3.3}, one can obtain
	  $$\frac{Eg(X_{\alpha,t}) -Eg(X_{\beta,t})}{\alpha-\beta}\to E\left[g(X_{\beta,t})\delta\left(\frac{Y_{\beta,t}DX_{\beta,t}}{\|DX_{\beta,t}\|_{L^2[0,T]}}\right)\right] \ \mbox{as} \ \alpha\to \beta.$$
		This completes the proof of Theorem \ref{dlc}.

\hfill$\square$



\end{document}